\documentclass[11pt]{article}

\usepackage{amssymb,amsmath}
\usepackage{amsthm}
\usepackage{hyperref}
\usepackage{mathrsfs}
\usepackage{textcomp}
\usepackage{verbatim}
\hypersetup{
    colorlinks,%
    citecolor=black,%
    filecolor=black,%
    linkcolor=black,%
    urlcolor=black
}

\usepackage{verbatim}

\makeatletter

\newdimen\bibspace
\renewenvironment{thebibliography}[1]{%
 \section*{\refname %or \bibname if you use ``book'' as the documentclass
       \@mkboth{\MakeUppercase\refname}{\MakeUppercase\refname}}%
     \list{\@biblabel{\@arabic\c@enumiv}}%
          {\settowidth\labelwidth{\@biblabel{#1}}%
           \leftmargin\labelwidth
           \advance\leftmargin\labelsep
           \itemsep\bibspace
           \parsep\z@skip     %
           \@openbib@code
           \usecounter{enumiv}%
           \let\p@enumiv\@empty
           \renewcommand\theenumiv{\@arabic\c@enumiv}}%
     \sloppy\clubpenalty4000\widowpenalty4000%
     \sfcode`\.\@m}
    {\def\@noitemerr
      {\@latex@warning{Empty `thebibliography' environment}}%
     \endlist}

\makeatother

\makeatletter

\newtheorem{thm}{Theorem}[section]
\newtheorem{lem}[thm]{Lemma}
\newtheorem{prop}[thm]{Proposition}
\newtheorem{defn}[thm]{Definition}

\newtheorem{cor}[thm]{Corollary}
\newtheorem{rem}[thm]{Remark}

\def\Xint#1{\mathchoice
  {\XXint\displaystyle\textstyle{#1}}%
  {\XXint\textstyle\scriptstyle{#1}}%
  {\XXint\scriptstyle\scriptscriptstyle{#1}}%
  {\XXint\scriptscriptstyle\scriptscriptstyle{#1}}%
  \!\int}
\def\XXint#1#2#3{{\setbox0=\hbox{$#1{#2#3}{\int}$}
  \vcenter{\hbox{$#2#3$}}\kern-.5\wd0}}

\def\dashint{\Xint-}

\newcommand{\al}{\alpha}                \newcommand{\lda}{\lambda}
\newcommand{\om}{\Omega}                \newcommand{\pa}{\partial}
\newcommand{\va}{\varepsilon}           \newcommand{\ud}{\mathrm{d}}
\newcommand{\be}{\begin{equation}}      \newcommand{\ee}{\end{equation}}
                 
\newcommand{\Lda}{\Lambda}              
\newcommand{\R}{\mathbb{R}}              \newcommand{\Sn}{\mathbb{S}^n}

\newcommand{\dlim}{\displaystyle\lim}
\newcommand{\dsup}{\displaystyle\sup}
\newcommand{\dmin}{\displaystyle\min}

\newcommand{\abs}[1]{\lvert#1\rvert}
\newcommand{\sbt}{\,\begin{picture}(-1,1)(-1,-3)\circle*{3}\end{picture}\ }

\begin{document}

\title{\textbf{The Nirenberg problem and its generalizations:\\ A unified approach}
\bigskip}

\author{\medskip Tianling Jin\footnote{Supported in part by NSF grant DMS-1362525.}, \  \ YanYan Li\footnote{Supported in part by NSF grants  DMS-1065971 and DMS-1203961.}, \ \
Jingang Xiong\footnote{Supported in part by the First Class Postdoctoral Science Foundation of China (No. 2012M520002) and Beijing Municipal Commission of Education for the Supervisor of Excellent Doctoral Dissertation (20131002701).}}

\date{\today}

\maketitle

\begin{abstract} 
Making use of integral representations, we develop a unified approach to establish blow up profiles, compactness and existence of positive solutions of the conformally invariant equations $P_\sigma(v)= Kv^{\frac{n+2\sigma}{n-2\sigma}}$ on the standard unit sphere $\mathbb{S}^n$ for all $\sigma\in (0,n/2)$, where $P_\sigma$ is the  intertwining operator of order $2\sigma$. Finding positive solutions of these equations is equivalent to seeking metrics in the conformal class of the standard metric on spheres with prescribed certain curvatures. When $\sigma=1$, it is the prescribing scalar curvature problem or the Nirenberg problem, and when $\sigma=2$, it is the prescribing $Q$-curvature problem.
 \end{abstract}

\tableofcontents

\section{Introduction}

\subsection{The main problem}

Let $\Sn$ be the $n$ dimensional unit sphere endowed with the induced metric $g_{\Sn}$ from $\R^{n+1}$.
For a given continuous function $K\ge 0$ on $\Sn$,  the aim of the paper is to study blow up profiles, compactness and existence of solutions of the equation
\be\label{main equ}
P_\sigma(v)=c(n,\sigma)K v^{\frac{n+2\sigma}{n-2\sigma}},\quad v>0 \quad \mbox{on }\Sn
\ee
where $n\ge 2, 0<\sigma<n/2$, $c(n,\sigma)=\Gamma(\frac{n}{2}+\sigma)/\Gamma(\frac{n}{2}-\sigma)$, $\Gamma$ is the Gamma function and $P_\sigma$ is an intertwining operator (see, e.g., Branson \cite{Br}) of  order $2\sigma$ as follows : 

\medskip

$\sbt\ $ $P_\sigma$ can  be viewed as the pull back operator of  the $\sigma$ power of the Laplacian $(-\Delta)^{\sigma}$ on $\R^n$ via the stereographic
projection:
\be \label{eq:r1}
(P_\sigma(\phi))\circ F=  |J_F|^{-\frac{n+2\sigma}{2n}}(-\Delta)^\sigma(|J_F|^{\frac{n-2\sigma}{2n}}(\phi\circ F))\quad \mbox{for }\phi\in C^2(\Sn),
\ee
where $F$ is the inverse of the stereographic projection and $|J_F|$ is the determinant of the Jacobian of $F$.

\medskip

$\sbt\ $ From a general Lie theoretic point of view, by Branson \cite{Br} $P_\sigma$ has the expression
\be\label{P sigma}
 P_\sigma=\frac{\Gamma(B+\frac{1}{2}+\sigma)}{\Gamma(B+\frac{1}{2}-\sigma)},\quad B=\sqrt{-\Delta_{g_{\Sn}}+\left(\frac{n-1}{2}\right)^2},
\ee
where $\Delta_{g_{\Sn}}$ is the Laplace-Beltrami operator on $(\Sn, g_{\Sn})$. Let $Y^{(k)}$ be a spherical harmonic of degree $k\ge 0$. Then we have
\be\label{eigenvalue on spherical harmonics}
 B\Big(Y^{(k)}\Big)=\left(k+\frac{n-1}{2}\right)Y^{(k)}\quad\mbox{and}\quad P_\sigma\Big(Y^{(k)}\Big)=\frac{\Gamma(k+\frac{n}{2}+\sigma)}{\Gamma(k+\frac{n}{2}-\sigma)}Y^{(k)}. 
\ee

\medskip

$\sbt\ $ The Green function of $P_{\sigma}$ is the spherical Riesz potential, i.e.,
\be\label{P sigma inverse}
P_\sigma^{-1}(f)(\xi)=c_{n,\sigma}\int_{\Sn}\frac{f(\zeta)}{\abs{\xi-\zeta}^{n-2\sigma}}\,\ud vol_{g_{\Sn}}(\zeta)\quad \mbox{for }f\in L^p(\Sn),
\ee
where $c_{n,\sigma}=\frac{\Gamma(\frac{n-2\sigma}{2})}{2^{2\sigma}\pi^{n/2}\Gamma (\sigma)}$, $p>1$ and $|\cdot|$ is the Euclidean distance in $\R^{n+1}$.

\medskip

%Let us recall a sharp Sobolev inequality that gives one reason why the existence of \eqref{main equ} does not follow directly from the standard variational method. 
The equation \eqref{main equ} involves critical exponent because of the Sobolev embeddings. Denote $H^{\sigma}(\mathbb{S}^n)$ as the $\sigma$ order fractional Sobolev space that consists of all functions $v\in L^2(\Sn)$ such that $(1-\Delta_{g_{\Sn}})^{\sigma/2}v\in L^2(\Sn)$, with the norm $\|v\|_{H^\sigma(\Sn)}:=\|(1-\Delta_{g_{\Sn}})^{\sigma/2}v\|_{L^2(\Sn)}$. The sharp Sobolev inequality on $\Sn$ (see Beckner \cite{Be}) asserts that
\be\label{pe1}
\left(\dashint_{\mathbb{S}^n}|v|^{\frac{2n}{n-2\sigma}}\,\ud vol_{g_{\Sn}}\right)^{\frac{n-2\sigma}{n}}\leq \frac{\Gamma(\frac{n}{2}-\sigma)}{\Gamma(\frac{n}{2}+\sigma)}
\dashint_{\mathbb{S}^n}vP_{\sigma}(v)\,\ud vol_{g_{\Sn}}\quad \mbox{for }v\in H^{\sigma}(\mathbb{S}^n),
\ee
and the equality holds if and only if $v$
has the form
\be \label{standard bubble on sn}
 v_{\xi_0, \lda}(\xi)=\left(\frac{2\lda}{2+(\lda^2-1)(1-\cos dist_{g_{\Sn}}(\xi, \xi_0))}\right)^{\frac{n-2\sigma}{2}}, \quad   \xi\in \Sn
\ee
for some $\xi_0\in \Sn$ and positive constant $\lda$, where $\dashint_{\mathbb{S}^n}=\frac{1}{|\mathbb{S}^n|}\int_{\mathbb{S}^n}$.
The inequality \eqref{pe1} can be proved by using \eqref{eq:r1} and
the sharp Sobolev inequality on $\R^n$ established by Lieb \cite{Lie83}
\be\label{eq:sharp on rn}
\left(\int_{\R^n}|u|^{\frac{2n}{n-2\sigma}}\,\ud x\right)^{\frac{n-2\sigma}{n}}\leq \frac{\Gamma(\frac{n}{2}-\sigma)}{\omega_n^{\frac{2\sigma}{n}}\Gamma(\frac{n}{2}+\sigma)}
\|u\|^2_{\dot H^{\sigma}(\mathbb{R}^n)}\quad \mbox{for }u\in \dot H^{\sigma}(\mathbb{R}^n),
\ee
where $\omega_n$ denotes the area of the $n$-dimensional unit sphere and $\dot H^{\sigma}(\mathbb{R}^n)$  is the closure of $C_c^{\infty}(\R^n)$ under the norm $\|u\|_{\dot H^{\sigma}(\R^n)}=\|(-\Delta)^{\sigma/2}u \|_{L^2(\R^n)}$.

%By \eqref{pe1}, the nonlinear term in equation \eqref{main equ} involves the critical growth.

\subsection{Conformal geometry and known results}

Equation \eqref{main equ} and its limiting case ($\sigma=n/2$)
\be \label{eq:limiting}
P_{n/2} w+(n-1)!=Ke^{nw} \quad \mbox{on }\Sn,
\ee
where
\[
P_{n/2}=\begin{cases}
\prod_{k=0}^{\frac{n-2}{2}}(-\Delta_{g_{\Sn}}+k(n-k-1)) \quad  &\mbox{for even }n\\
(-\Delta_{g_{\Sn}}+(\frac{n-1}{2})^2)^{1/2} \prod_{k=0}^{\frac{n-3}{2}}(-\Delta_{g_{\Sn}}+k(n-k-1)) \quad &\mbox{for odd }n,
\end{cases}
\] are closely related to several problems in conformal geometry.

\medskip

\noindent \textbf{(i) The Nirenberg problem ($\sigma=1$).} The classical Nirenberg problem asks: \textit{Which function $K$ on $(\Sn,g_{\Sn})$ is the scalar curvature (Gauss curvature in dimension $n=2$) of a metric $g$ that is conformal to $g_{\Sn}$?} If we denote $g=e^{2w}g_{\Sn}$, then the problem is equivalent to solving
\[
P_1w+1=-\Delta_{g_{\Sn}}w+1=Ke^{2w} \quad\mbox{on } \mathbb{S}^2,
\]
and
\[
P_1u=-\Delta_{g_{\Sn}}v+c(n)R_0v=c(n)Kv^{\frac{n+2}{n-2}} \quad\mbox{on } \Sn\mbox{ for }n\geq 3,
\]
where $c(n)=(n-2)/(4(n-1)), R_0=n(n-1)$ is the scalar curvature of $(\Sn, g_{\Sn})$ and $v=e^{\frac{n-2}{2}w}$.

The first work on the problem is by D. Koutroufiotis \cite{K}, where the solvability on $\mathbb{S}^2$ is established
when $K$ is assumed to be an antipodally symmetric function which is close to $1$. Moser \cite{Ms}
established the solvability on $\mathbb{S}^2$ for all antipodally symmetric functions $K$ which is positive somewhere. Without assuming any symmetry assumption on $K$, sufficient conditions were given in dimension
$n=2$ by Chang and Yang \cite{CY87} and \cite{CY88}, and in dimension $n=3$ by Bahri and Coron \cite{BC}.
Compactness of all solutions in dimensions $n=2,3$ can be found in work of Chang, Gursky and Yang \cite{CGY}, Han \cite{H1} and Schoen and Zhang \cite{SZ}. In these dimensions, a sequence of solutions
can not blow up at more than one point. Compactness and existence of solutions in higher dimensions were studied by Li in \cite{Li95} and \cite{Li96}.
The situation is very different, as far as the compactness issues are concerned: In dimension $n\geq 4$,
a sequence of solutions can blow up at more than one point, as shown in \cite{Li96}. There have been many papers on the problem and related ones. Without any attempt to give a complete list of references, one may see, for example, \cite{AB1, AB2, AH, B1, BLR, BA, BrN, By, CNY, CZ, CL,CY91, CL2, CD,  ChLi, CX, D, EM, H1, HL1, HL2, J1, J2, Li93, LL, M, S1, S2, WY, Zhang, Zhu} and so on.

\medskip

\noindent \textbf{(ii) Prescribed $Q$-curvature problem ($\sigma=2$).}
For a smooth compact Riemannian manifold $(M,g)$ of dimension $n\ge 4$,
the fourth order conformally invariant \emph{Paneitz operator} and Branson's \emph{$Q$-curvature} are given by
%\[
%\begin{split}
%P^g_2 &= \Delta_g^2 -\mathrm{div}_g(\frac23 R_g g-2Ric_g)d \\
%Q_g&=-\frac{1}{12}(\Delta_g R_g-R_g^2+3|Ric_g|^2)
%\end{split}
%\]
%for $n=4$ and
\[
\begin{split}
P^g_2&=  \Delta_g^2 -\mathrm{div}_g(a_n R_g g+b_nRic_g)d+\frac{n-4}{2}Q_g \\
Q_g&=-\frac{1}{2(n-1)} \Delta_g R_g+ c_n R_g^2-\frac{2}{(n-2)^2} |Ric_g|^2,
\end{split}
\]
where $R_g$ and $Ric_g$ denote the scalar curvature and Ricci tensor of $g$ respectively, and
$a_n=\frac{(n-2)^2+4}{2(n-1)(n-2)}$, $b_n=-\frac{4}{n-2}$, $c_n=\frac{n^3-4n^2+16n-16}{8(n-1)^2(n-2)^2}$. 
As in the Nirenberg problem, the problem of prescribing $Q$-curvature in the conformal class of $g_{\Sn}$ is reduced to the study of existence of positive solutions to equations \eqref{eq:limiting} and \eqref{main equ} with $\sigma=2$.  On $\mathbb{S}^4$, Moser's type result was obtained by Brendle \cite{B'} via a flow method (see also Baird-Fardoun-Regbaoui \cite{BFR1}).  Without  assuming any symmetry assumptions, sufficient conditions were given by   Wei-Xu \cite{WX2}, Brendle \cite{B}, Malchiodi-Struwe \cite{MSt}. On $\Sn$ of dimension $n\ge 5$, the problem has been studied by Djadli-Hebey-Ledoux \cite{DHL}, Robert \cite{Ro}, Djadli-Malchiodi-Ahmedou \cite{D1, D2}, Felli \cite{Felli} and many others. For the constant $Q$-curvature problem on general Riemannian manifolds and related work, one may refer to Chang-Yang \cite{CY95}, Chang-Gursky-Yang \cite{CGY2}, Esposito-Robert \cite{ER}, Hebey-Robert-Wen \cite{HRW}, Qing-Raske \cite{QR2}, Djadli-Malchiodi \cite{DM}, Ndiaye \cite{NCB}, Weinstein-Zhang \cite{WZhang}, Hebey-Robert \cite{HebeyRobert}, Li-Li-Liu \cite{LLL}, Gursky-Malchiodi \cite{GM}, Hang-Yang \cite{HY1,HY2,HY3} and the references therein.

\medskip

\noindent \textbf{(iii) Prescribed higher order and fractional order curvature problems.}  Higher order and fractional order conformally invariant differential (or pseudodifferential) operators on Riemannian manifolds have also been studied. In \cite{GJMS},  Graham, Jenne, Mason and Sparling constructed a sequence of conformally invariant elliptic operators, which are called GJMS operators nowadays,  based on ambient metric construction of Fefferman-Graham \cite{FG}. Recently, an explicit formula and a recursive formula for GJMS operators and $Q$-curvatures have been found by Juhl \cite{Juhl1,Juhl2} (see also Fefferman-Graham \cite{FG13}). In \cite{GZ}, Graham and Zworski introduced a family of fractional order conformally invariant operators on the conformal infinity of asymptotically hyperbolic manifolds via scattering theory, where the operators $P_\sigma$ (or more precisely, $P_\sigma^{g_{\Sn}}$) on $\Sn$ in equations \eqref{main equ} and \eqref{eq:limiting} are the most typical examples. Some new interpretation and properties of those fractional operators and their associated fractional $Q$-curvatures were later given  by Chang-Gonz\'alez \cite{CG} and Case-Chang \cite{CaC}, which also provide geometric interpretation of the extension formulations for fractional Laplacian $(-\Delta)^\sigma$ established by Caffarelli-Silvestre \cite{CaffS} when $\sigma\in (0,1)$ and by R. Yang \cite{Y} when $\sigma>1$. 

Prescribing such fractional $Q$-curvature of order $2\sigma$ on $\Sn$ can be considered as generalizations and extensions of the Nirenberg problem and the prescribed $Q$-curvature problem. One may see the work of, among many others, Escobar \cite{Escb}, Chang-Xu-Yang \cite{CXY}, Escobar-Garcia \cite{EG}, Han-Li \cite{HL99}, Djadli-Malchiodi-Ahmedou \cite{D3} for $\sigma=1/2$; Jin-Li-Xiong \cite{JLX, JLX2}, Chen-Liu-Zheng \cite{CLZ}, Fang \cite{FangFei}, Abdelhedi-Chtioui-Hajaiej \cite{AC} for $\sigma\in(0,1)$; Wei-Xu \cite{WX, WX2}, Brendle \cite{B}, Chen-Xu \cite{CX'}, Baird-Fardoun-Regbaoui \cite{BFR2}  for $2<\sigma=n/2$; 
 and Zhu \cite{Zhu14} for $\sigma>n/2$ in terms of integral equations.  Other work which are closely related to these fractional $Q$-curvatures includes: prescribing fractional order curvature problems on general manifolds by Qing-Raske \cite{QR}, Gonz\'alez-Mazzeo-Sire \cite{GMS}, Gonz\'alez-Qing \cite{GQ}; fractional Yamabe flows and weighted trace inequalities by Jin-Xiong \cite{JX1, JX2};  analysis of local solutions of \eqref{main equ} near isolated singularities by Caffarelli-Jin-Sire-Xiong\cite{CJSX}; some Liouville theorem for indefinite fractional problem by Chen-Zhu \cite{CZhu}; etc.

\subsection{Our main results}

This paper is devoted to equation \eqref{main equ} for all \[\sigma\in(0,n/2).\]

First of all, equation \eqref{main equ} is not always solvable. Indeed, we have the Kazdan-Warner type obstruction: For any conformal Killing vector field $X $ on $\Sn$, there holds
\be \label{eq:KZ}
\int_{\Sn}(\nabla _X K) v^{\frac{2n}{n-2\sigma}}\,\ud vol_{g_{\Sn}}= 0
\ee
for any solution $v$ of \eqref{main equ}, see \cite{BE} and \cite{Xu}. Hence, if $K(\xi)=\xi_{n+1}+2$ for example, then equation \eqref{main equ} has no positive solutions.

\begin{defn}
For $d>0$, we say that $K\in C(\Sn)$ has flatness order greater than $d$ at a point $\xi\in\Sn$ if, in a geodesic coordinate system
$\{y_1,\cdots, y_n\}$ centered at $\xi$,
there exists a neighborhood $\mathscr{O}$ of $0$ such that $K(y)=K(0)+o(|y|^{d})$ in $\mathscr{O}$.
\end{defn}

\begin{thm}\label{K-M-E-S} Let $n\ge 2$, $0<\sigma<n/2$, and $K>0$ be a continuous antipodally symmetric function on $\Sn$, i.e., $K(\xi)=K(-\xi)$ $\forall~\xi\in \Sn$.
If there exists a maximum point of $K$ at which $K$ has flatness order greater than $n-2\sigma$,
then \eqref{main equ} has at least one positive solution in $C^{2\sigma^*}$, where $\sigma^*=\sigma$ if $2\sigma \notin \mathbb{N}^+$ and otherwise $0<\sigma^*<\sigma$.
\end{thm}

As we mentioned before, Theorem \ref{K-M-E-S} was known for $\sigma$ in some regions: see Escobar-Schoen \cite{ES} for $\sigma=1$, Djadli-Hebey-Ledoux \cite{DHL} for $\sigma=2$, Robert  \cite{Ro10} for $\sigma$ being other integers, and Jin-Li-Xiong \cite{JLX2} for $0<\sigma<1$. If $K$ is $C^{1}$ and $n-2\sigma\le 1$, then at its maximum points $K$ automatically has flatness order greater than $n-2\sigma$. Therefore, we have

\begin{cor}
Suppose $(n-1)/2\le\sigma<n/2$. Then for any antipodally symmetric positive function $K\in C^{1}$ on $\Sn$, equation \eqref{main equ} has a $C^{2\sigma^*}$ solution.
\end{cor}

In the next theorem, we consider the functions $K$ of the following form (without symmetry assumptions): for any critical point
$\xi_0$ of $K$, in some geodesic normal coordinates $\{y_1, \cdots, y_n\}$ centered at $\xi_0$, there exist some small neighborhood $\mathscr{O}$ of $0$ and positive constant $\beta=\beta(\xi_0)$ such that 
\be\label{eq:form of K}
 K(y)=K(0)+\sum_{j=1}^{n}a_{j}|y_j|^{\beta}+R(y) \quad \mbox{in } \mathscr{O},
\ee
where $R(y)\in C^{[\beta]-1,1}(\mathscr{O})$ satisfies  $\sum_{s=0}^{[\beta]}|\nabla^sR(y)||y|^{-\beta+s} \to 0$ as $y\to 0$ and $[\beta]$ is the integer part of $\beta$.

\begin{thm}
 \label{main thm A} Let $n \geq 2$ and $0<\sigma<n/2$. Suppose that $K\in C^{1}(\Sn)$ ($K\in C^{1,1}(\Sn)$ if $0<\sigma\le 1/2$) is a positive function satisfying that for any critical point
$\xi_0$ of $K$, in some geodesic normal coordinates $\{y_1, \cdots, y_n\}$ centered at $\xi_0$, there exist
some small neighborhood $\mathscr{O}$ of $0$ and positive constants $\beta=\beta(\xi_0)\in (n-2\sigma,n)$, $\gamma\in (n-2\sigma, \beta]$
such that $K\in C^{[\gamma],\gamma-[\gamma]}(\mathscr{O})$ satisfying \eqref{eq:form of K} with $a_j=a_j(\xi_0)\neq 0$ for every $j$ and $\sum_{j=1}^n a_j\neq 0$.
If
\[
\sum_{\xi\in \Sn\mbox{ such that }\nabla_{g_{\Sn}}K(\xi)=0,\  \sum_{j=1}^na_j(\xi)<0}(-1)^{i(\xi)}\neq (-1)^n,
\]
where
\[
 i(\xi)=\#\{a_j(\xi): \nabla_{g_{\Sn}}K(\xi)=0,a_j(\xi)<0,1\leq j\leq n\},
\]
then \eqref{main equ} has at least one $C^2$ positive solution. Moreover, there exists a positive constant $C$ depending only on $n, \sigma$ and $K$ such that for all positive $C^2$ solutions $v$ of \eqref{main equ},
\[
1/C\leq v\leq C\quad\mbox{and}\quad\|v\|_{C^2(\Sn)}\leq C.
\]
\end{thm}

For $\sigma=1, n=3$, the existence part of the above theorem was established by Bahri and Coron \cite{BC}, and the compactness part were proved in Chang, Gursky and Yang \cite{CGY} and Schoen and Zhang \cite{SZ}. For $\sigma=1, n\geq 4$, the above theorem was proved by Li \cite{Li95}. For $\sigma=2$ and $n\ge 5$, it was proved by Djadli-Malchiodi-Ahmedou \cite{D1, D2} and Felli \cite{Felli}.  For $\sigma\in (0,1)$, it was recently proved by Jin-Li-Xiong \cite{JLX, JLX2}.

\begin{defn}\label{def*} For any real number $\beta> 1$, we say that a sequence of functions
$\{K_i\}$ satisfies condition $(*)'_{\beta}$ for some sequence of constants
$L(\beta,i)$ in some region $\om_i$, if $\{K_i\}\in C^{[\beta],\beta-[\beta]}(\om_i)$ satisfies
\[
[\nabla^{[\beta]} K_i]_{C^{\beta-[\beta]}(\om_i)}\leq L(\beta,i),
\]
and, if $\beta\geq 2$, that
\[
|\nabla^sK_i(y)|\leq L(\beta,i)|\nabla K_i(y)|^{(\beta-s)/(\beta-1)},
\]
for all $2\leq s\leq [\beta]$, $y\in \om_i$, $\nabla K_i(y)\neq 0$.
\end{defn}

Note that the function $K$ in Theorem \ref{main thm A} satisfies $(*)'_{\beta}$ condition.
%We have the a priori estimate in the energy space:

\begin{thm}
 \label{main thm B} Let $n \geq 2$ and $0<\sigma<n/2$, and $K\in C^{1}(\Sn)$ ($K\in C^{1,1}(\Sn)$ if $0<\sigma\le 1/2$)  be a positive function.
If there exists some constant $d>0$ such that $K$ satisfies $(*)'_{(n-2\sigma)}$ for some constant $L>0$ in $\om_{d}:=\{\xi\in \Sn:|\nabla_{g_0}K(\xi)|<d\}$, then for any positive solution $v\in C^2(\Sn)$
of \eqref{main equ},
\be \label{finite energy}
 \|v\|_{H^{\sigma}(\Sn)}\leq C,
\ee
where $C$ depends only on $n,\sigma, \inf_{\Sn}K>0, \|K\|_{C^1(\Sn)}(\|K\|_{C^{1,1}(\Sn)} \mbox{ if } \ 0<\sigma\le \frac 12), L$, and $d$.
\end{thm}

The above theorem was proved by Schoen-Zhang \cite{SZ} for $n=3$ and $\sigma=1$, by Li \cite{Li95} for $n\ge 4$ with $\sigma=1$, by Jin-Li-Xiong \cite{JLX} for $\sigma\in (0,1)$, and by Djadli-Malchiodi-Ahmedou \cite{D1, D2}  and Felli \cite{Felli} for $\sigma=2$.

\begin{thm}
 \label{main thm C}  Let $n \geq 2$ and $0<\sigma<n/2$. Suppose that $\{K_i\}\in C^{1}(\Sn)$ ($K\in C^{1,1}(\Sn)$ if $0<\sigma\le 1/2$)  is a sequence of positive functions with uniform $C^1$ ($C^{1,1}$ if  $0<\sigma\le 1/2$) norm and $1/A_1\leq K_i\leq A_1$ on $\Sn$ for some $A_1>0$ independent of $i$. Suppose also that $\{K_i\}$ satisfying
$(*)'_{\beta}$ condition for some constants $\beta\in(n-2\sigma,n)$,  $L, d>0$ in $\om_{d}$.
Let $\{v_i\}\in C^2(\Sn)$ be a sequence of corresponding positive solutions of \eqref{main equ} and $v_i(\xi_i)=\max_{\Sn}v_i$ for some $\xi_i$.
Then, after passing to a subsequence, $\{v_i\}$ is either bounded
in $L^\infty(\Sn)$ or blows up at exactly one point in the strong sense: There exists a sequence of M\"obius diffeomorphisms $\{\varphi_i\}$ from
$\Sn$ to $\Sn$ satisfying $\varphi_i(\xi_i)=\xi_i$ and $|\det \ud \varphi_i(\xi_i)|^{\frac{n-2\sigma}{2n}}=v_i^{-1}(\xi_i)$ such that
\[
 \|T_{\varphi_i}v_i -1\|_{C^0(\Sn)}\to 0,\quad \text{as }i \to \infty,
\] where $T_{\varphi_i}v_i:=(v\circ\varphi_i)|\det \ud \varphi_i|^{\frac{n-2\sigma}{2n}}$.
\end{thm}

For $\sigma=1, n=3$, the above theorem was established by Chang-Gursky-Yang \cite{CGY} and by Schoen and Zhang in \cite{SZ}. For $\sigma=1, n\ge4$, the above theorem was proved by Li \cite{Li95}. For $\sigma=2$, it was proved by Djadli-Malchiodi-Ahmedou \cite{D1, D2} and Felli \cite{Felli}. For $\sigma\in (0,1)$, the above theorem was proved by Jin-Li-Xiong \cite{JLX}.

\medskip

By the stereographic projection and Green's representation, we write equation \eqref{main equ} as the form
\be\label{eq:main equ on rn}
u(x)= \int_{\R^n} \frac{K(y)u(y)^{\frac{n+2\sigma}{n-2\sigma}}}{|x-y|^{n-2\sigma}}\,\ud y.
\ee
If  $K=1$, it follows from Chen-Li-Ou \cite{CLO} or Li \cite{Li04} that $u$ has to be of the form
\[
c\left(\frac{\lda}{1+\lda^2|x-x_0|^2}\right)^\frac{n-2\sigma}{2}, \quad \mbox{for } c, \lda>0, ~ x_0\in \R^n,
\]
which we call a bubble for $\lda$ large.

The main idea of our unified approach for all $\sigma\in (0,n/2)$ is to carry out blow up analysis as in Schoen-Zhang \cite{SZ} and Li \cite{Li95} for nonnegative solutions of the nonlinear integral equation \eqref{eq:main equ on rn} without using extension formulations for the fractional Laplacian, where some brief descriptions are as follows.  In Section \ref{section2}, we first establish some  local estimates and Harnack inequality for nonnegative solutions of linear integral equations. Then we consider positive solutions of nonlinear integral equations with an isolated simple blow up point (Definition \ref{def4.2}). We prove that those solutions are close to some bubbles in $C^2$ norm in a small neighborhood of the blow up point (Proposition \ref{prop:blow up a bubble}) and satisfy a sharp upper bound in a ball with fixed size (Proposition \ref{prop:upbound2}). At last, given some flatness condition on $K$, we are able to show that isolated blow up points have to be isolated simple blow up points. All the analysis techniques for integral equations needed in this paper are developed. In Section \ref{section3}, we follow the arguments of Li \cite{Li95} and prove our compactness results:  Theorem \ref{main thm B}, Theorem \ref{main thm C} and the compactness part of Theorem \ref{main thm A}. In Section \ref{section4}, we prove our existence results: Theorem \ref{K-M-E-S} and the existence part of Theorem \ref{main thm A}. %Even though the same arguments of Jin-Li-Xiong \cite{JLX2} can give a proof of the existence part of Theorem \ref{main thm A}, we prove it here using a degree argument by constructing a homotopy along which $\sigma$ is deformed to $1$.

\bigskip

\noindent\textbf{Acknowledgements:}  %T. Jin was supported in part by NSF grant DMS-1362525. Y.Y. Li was supported in part by NSF grants  DMS-1065971 and DMS-1203961. J. Xiong was supported in part by the First Class Postdoctoral Science Foundation of China (No. 2012M520002) and Beijing Municipal Commission of Education for the Supervisor of Excellent Doctoral Dissertation (20131002701). 
Part of this work was done while J. Xiong was visiting Universit\'e Paris VII and Universit\'e Paris XII in November 2013 through a Sino-French  research program in mathematics. He would like to thank Professors Yuxin Ge and Xiaonan Ma for the arrangement. He also thanks Professor Gang Tian for his support and encouragement. T. Jin would like to thank Professors Henri Berestycki and Luis Silvestre for their support and encouragement. The authors thank Professor Hongjie Dong for suggesting the current proof of Proposition \ref{prop:har}.

\section{Local analysis}
\label{section2}

\subsection{Linear integral equations}\label{sec:linear integral equations}

The regularity properties and local estimates of nonnegative solutions of Schr\"odinger equation $-\Delta u+Vu=0$ play important roles in the study of the Nirenberg problem and Yamabe problem. In this subsection, we will develop similar results for linear integral equations.

The first result was proved in Li \cite{Li04}, which is a Brezis-Kato \cite{BK} type estimate.

\begin{thm}[Theorem 1.3 in \cite{Li04}] \label{thm:li} For $n\ge 1, 0<\sigma<n/2$, and $\nu>r>n/(n-2\sigma)$, there exist positive constants $\bar\delta<1$ and $C\ge 1$, which depend only on $n,\sigma, r$ and $\nu$, such that for every $0\le V\in L^{n/2\sigma}(B_3)$ with $\|V\|_{L^{n/2\sigma}(B_3)}\le \bar \delta$, $h\in L^\nu(B_2)$, and $0\le u\in L^r(B_3)$ satisfying
\be\label{eq:li}
u(x)\leq \int_{B_3} \frac{V(y)u(y)}{|x-y|^{n-2\sigma}}\,\ud y+h(x)\quad x\in B_2,
\ee
we have $u\in L^\nu(B_1)$ and
\[
\|u\|_{L^\nu(B_{1})}\leq C(\|u\|_{L^r(B_3)}+\|h\|_{L^\nu(B_2)}).
\]
\end{thm}
Consequently, we have
\begin{cor}\label{cor:li} For $n\ge 1, 0<\sigma<n/2$, $r>n/(n-2\sigma)$ and $p>n/2\sigma$, let $0\le V\in L^p(B_3)$, $h\in L^\infty(B_2)$ and $0\leq u\in L^r(B_3)$ satisfy \eqref{eq:li}.  Then $u\in L^\infty(B_1)$ and
\[
\|u\|_{L^\infty(B_1)}\leq C(\|u\|_{L^r(B_3)}+\|h\|_{L^\infty(B_2)}),
\]
where $C$ depends only on $n,\sigma$, $r,p$ and an upper bound of $\|V\|_{L^p(B_3)}$. Moreover, if in addition that $h\in C^0(B_1)$ and it satisfies
\be\label{eq:li eq}
u(x)=\int_{B_3} \frac{V(y)u(y)}{|x-y|^{n-2\sigma}}\,\ud y+h(x) \quad \mbox{in }B_2,
\ee 
then $u\in C^0(B_1)$.
\end{cor}
\begin{proof}
The first part follows from Theorem \ref{thm:li} directly. For the second part, if $u\in L^\infty(B_1)$, then
\[
\int_{B_3} \frac{V(y)u(y)}{|x-y|^{n-2\sigma}}\,\ud y=\int_{B_1} \frac{V(y)u(y)}{|x-y|^{n-2\sigma}}\,\ud y+\int_{B_3\setminus B_1} \frac{V(y)u(y)}{|x-y|^{n-2\sigma}}\,\ud y
\]
is continuous in $B_1$. Since $h\in C^0(B_1)$, then it follows from \eqref{eq:li eq} that $u\in C^0(B_1)$.
\end{proof}

The next result is about a Harnack inequality for the integral equation \eqref{eq:li eq}.

\begin{prop}\label{prop:har} For $n\ge 1, 0<\sigma<n/2$, $r>n/(n-2\sigma)$ and $p>n/2\sigma$, let $0\le V\in L^{p}(B_3)$, $0\le h \in C^0(B_2)$  and $0\leq u\in L^r(B_3)$ satisfy \eqref{eq:li eq}.  If there exists a constant $c_0\ge 1$ such that $
\max_{\overline B_1}h \le c_0 \min_{\overline B_1} h,$
then
\[
\max_{\overline B_1} u\le C\min_{\overline B_1} u,
\]
where $C>0$ depends only on $n,\sigma, c_0, p$ and an upper bound of $\|V\|_{L^p(B_3)}$.
\end{prop}

\begin{proof} For any fixed $1<s<\frac{n}{n-2\sigma}$ and $t>1$ satisfying  $\frac{1}{p}+\frac1s+\frac1t=1$, we can choose $0<\delta<1$ such that  for any $x\in B_2$ there holds
\be \label{eq:choice of t}
\begin{split}
&\int_{B_\delta(x)} |x-y|^{2\sigma-n} V(y)\,\ud y  \\
&\le \|V\|_{L^p(B_{\delta}(x))} \||x-\cdot|^{2\sigma-n} \|_{L^s(B_\delta(x))}|B_\delta|^{1/t} \\
&<1/2.
\end{split}
\ee
Meanwhile, by Corollary \ref{cor:li} and standard translation arguments, $u$ is continuous in $B_{3/2}$.

To illustrate  the method, we first consider the simplest case by assuming that $u(0)=\max_{\overline B_1} u$ and $u(\bar x)=\min_{\overline B_1} u$ for some $\bar x \in \overline B_1$.
By the equation \eqref{eq:li eq} of $u$ and H\"older inequality, we have
\[
\begin{split}
u(0) &=\int_{B_3} \frac{V(y)u(y)}{|y|^{n-2\sigma}}\,\ud y+h(0)\\
&\le u(0)\int_{B_\delta} \frac{V(y)}{|y|^{n-2\sigma}}\,\ud y+\int_{B_3\setminus B_\delta}  \frac{V(y)u(y)}{|y|^{n-2\sigma}}\,\ud y+h(0)\\
& \le \frac{1}{2} u(0) +\int_{B_3\setminus B_\delta} \frac{V(y)u(y)}{|y|^{n-2\sigma}}\,\ud y +c_0 h(\bar x).
\end{split}
\]
Note that $|\bar x-y|^{n-2\sigma}/|y|^{n-2\sigma} \le C_0(n,\sigma, \delta)<\infty$ for $y\in B_3\setminus B_\delta$. Hence,
\[
\frac12 u(0) \le C_0\int_{B_3\setminus B_\delta} \frac{V(y)u(y)}{|\bar x-y|^{n-2\sigma}}\,\ud y +c_0 h(\bar x) \le (C_0+c_0) u(\bar x).
\]

In general, let
\[
M_k=\max_{\overline B_{1-2^{-k}}} u,\quad m_k=\min_{\overline B_{1-2^{-k}}} u
\]
Since $u$ is continuous, we may assume $M_k=u(z'_k)$ and $m_k=u(z_k'')$ for $z_k',z_k''\in \overline B_{1-1/2^k}$.
By the equation \eqref{eq:li eq} of $u$ and H\"older inequality, we have
\[
\begin{split}
 &M_k-h(z_k')\\
 &=u(z_k')-h(z_k')\\
&=\int_{B_\delta(z_k')\cap B_{1-1/2^{k+1}}}\frac{V(y)u(y)}{|z_k'-y|^{n-2\sigma}}\,\ud y+ \int_{B_3\setminus (B_\delta(z_k')\cap B_{1-1/2^{k+1}})}  \frac{V(y)u(y)}{|z_k'-y|^{n-2\sigma}}\,\ud y\\
&=\int_{B_\delta(z_k')\cap B_{1-1/2^{k}}}\frac{V(y)u(y)}{|z_k'-y|^{n-2\sigma}}\,\ud y+\int_{B_\delta(z_k')\cap (B_{1-1/2^{k+1}}\backslash B_{1-1/2^{k}})}\frac{V(y)u(y)}{|z_k'-y|^{n-2\sigma}}\,\ud y\\
&\quad+ \int_{B_3\setminus (B_\delta(z_k')\cap B_{1-1/2^{k+1}})}  \frac{V(y)u(y)}{|z_k'-y|^{n-2\sigma}}\,\ud y\\
&\leq \frac{M_k}{2}+\frac{c_1}{2} 2^{-(k+1)\va } M_{k+1}+ \int_{B_3\setminus (B_\delta(z_k')\cap B_{1-1/2^{k+1}})}  \frac{V(y)u(y)}{|z_k'-y|^{n-2\sigma}}\,\ud y,
\end{split}
\]
where $\va=\frac{1}{t}$ for $t>1$ in \eqref{eq:choice of t} and $c_1$ is a positive constant depending on $n,\sigma, p$, an upper bound of $ \|V\|_{L^p(B_2)}$ and the $s,t$ in \eqref{eq:choice of t}. Thus,
\[
\int_{B_3\setminus (B_\delta(z_k')\cap B_{1-1/2^{k+1}})}  \frac{V(y)u(y)}{|z_k'-y|^{n-2\sigma}} \,\ud y+h(z_k')\geq \frac12 (M_k-c_1 2^{-(k+1)\va} M_{k+1}).
\]
By the equation, we also have
\[
\begin{split}
m_k&=u(z_k'')\\&\ge  \int_{B_3\setminus (B_\delta(z_k')\cap B_{1-1/2^{k+1}})}  \frac{V(y)u(y)}{|z_k''-y|^{n-2\sigma}}\,\ud y+h(z_k'')\\ &
\ge \int_{B_3\setminus (B_\delta(z_k')\cap B_{1-1/2^{k+1}})} \frac{|z_k'-y|^{n-2\sigma}}{|z_k''-y|^{n-2\sigma}} \frac{V(y)u(y)}{|z_k'-y|^{n-2\sigma}}\,\ud y+\frac{1}{c_0}h(z_k')\\&
\geq 2 c_2 2^{-(k+1)(n-2\sigma)} \left(\int_{B_3\setminus(B_\delta(z_k')\cap B_{1-1/2^{k+1}})} \frac{V(y)u(y)}{|z_k'-y|^{n-2\sigma}}\,\ud y
+h(z_k') \right)\\&
\geq c_2 2^{-(k+1)(n-2\sigma)}(M_k-c_1 2^{-(k+1)\va}M_{k+1}),
\end{split}
\]
where we  used that
\[
\begin{split}
\frac{|z_k'-y|^{n-2\sigma}}{|z_k''-y|^{n-2\sigma}}&\ge 2 c_2 2^{-(k+1)(n-2\sigma)} \quad\forall \ y\in B_3\setminus (B_\delta(z_k')\cap B_{1-1/2^{k+1}})\\
\frac{1}{c_0} &\ge 2 c_2 2^{-(k+1)(n-2\sigma)}
\end{split}
\]
for some $c_2>0$ depending only on $n,\sigma, \delta$ and $c_0$. Hence, we have
\[
m_k\ge c_2 2^{-(k+1)(n-2\sigma)}(M_k-c_1 2^{-(k+1)\va} M_{k+1}).
\]
We then multiply it by $2^{(k+1)(n-2\sigma)-\va k(k+1)/2}c_1^k$ and sum in $k$, we have
\[
\sum_{k=1}^\infty 2^{(k+1)(n-2\sigma)-\va k(k+1)/2}c_1^k m_k
\ge \sum_{k=1}^\infty c_2\left(2^{-\va k(k+1)/2}c_1^kM_k- 2^{-\va (k+1)(k+2)/2}c_1^{k+1}M_{k+1}\right).
\]
The right-hand side equals to $c_1c_2 2^{-\va} M_1$, and the left-hand side is that
\[
\sum_{k=1}^\infty 2^{(k+1)(n-2\sigma)-\va k(k+1)/2} c_1^k m_k \le m_1 \sum_{k=1}^\infty 2^{(k+1)(n-2\sigma)-\va k(k+1)/2} c_1^k\le C m_1.\]
Therefore, the Harnack inequality is proved.
\end{proof}

\begin{rem} In the proof of Proposition \ref{prop:har}, the constant $C$ in Harnack inequality will be unchanged if the integral domain $B_3$ is replaced by any open set $\om$ with $B_3\subset \om$. Also $|x-y|^{2\sigma-n}$ can be replaced by a kernel function $G(x,y)$ which satisfies $G(x,y)=G(y,x)$
\[
\Lda^{-1}\le |x-y|^{n-2\sigma}G(x,y)\le \Lda, \quad  |\nabla_x G(x,y)| \le \Lda |x-y|^{2\sigma-n-1}
\]
for some constant $\Lda\ge 1$.
\end{rem}

We will have H\"older estimates of $u$ satisfying \eqref{eq:li eq} if $h$ is in addition H\"older continuous.

\begin{prop}\label{prop:holder} For $n\ge 1, 0<\sigma<n/2$, $r>n/(n-2\sigma)$, $p>n/2\sigma$ and $\bar\alpha\in (0,1)$, let $0\le V\in L^{p}(B_3)$, $h \in C^{\bar\alpha}(B_2)$  and $0\leq u\in L^r(B_3)$ satisfy \eqref{eq:li eq}. Then $u\in C^\al(B_1)$ and 
\[
\|u\|_{C^\al(B_1)}\le C(\|u\|_{L^r(B_3)}+\|h\|_{C^{\bar\al}(B_2)}),
\]
where $C>0$ and $\alpha$ depends only on $n,\sigma, c_0, p, \bar\alpha$ and an upper bound of $\|V\|_{L^p(B_3)}$.
\end{prop}

\begin{proof}
It follows from Corollary \ref{cor:li} that
\[
\|u\|_{ L^\infty(B_{1.9})}\le C(\|u\|_{L^r(B_3)}+\|h\|_{L^\infty(B_2)}).
\]

For $x, z\in B_1$ satisfying $|x-z|<1/8$, we have
\[
\begin{split}
u(x)-u(z)&= \int_{B_1(x)\cap B_1(z)} \left(\frac{1}{|x-y|^{n-2\sigma}}-\frac{1}{|z-y|^{n-2\sigma}}\right)V(y)u(y)\,\ud y\\
&+\int_{B_3\setminus (B_1(x)\cap B_1(z))} \left(\frac{1}{|x-y|^{n-2\sigma}}-\frac{1}{|z-y|^{n-2\sigma}}\right)V(y)u(y)\,\ud y\\
&+h(x)-h(z)\\
&=I+II+III.
\end{split}
\]
By the mean value theorem, we have
\[
\begin{split}
|II|& \le C|x-z|\int_{B_3\setminus (B_1(x)\cap B_1(z))} \left(\frac{1}{|x-y|^{n-2\sigma+1}}\right)V(y)u(y)\,\ud y\\
& \le C|x-z|\int_{B_3\setminus (B_1(x)\cap B_1(z))} \left(\frac{1}{|x-y|^{n-2\sigma}}\right)V(y)u(y)\,\ud y\\
&\le  C |x-z|(u(x)-h(x))\\
&\le C (\|u\|_{L^r(B_3)} + \|h\|_{L^\infty(B_2)})|x-z|.
\end{split}
\]
For $III$, we have
\[
|III|\le \|h\|_{C^{\bar\alpha}(B_1)}|x-z|^{\bar\alpha}.
\]
To estimate $I$, we split the integral into the following four subdomains:
\[
\begin{split}
&U_1:= B_1(x)\cap B_1(z) \cap \{y: |y-x|<|x-z|/2\}\\
&U_2:= B_1(x)\cap B_1(z) \cap \{y: |y-z|<|x-z|/2\}\\
&U_3:= B_1(x)\cap B_1(z) \cap \{y: |y-z|>|y-x|>|x-z|/2\}\\
&U_4:= B_1(x)\cap B_1(z) \cap \{y: |y-x|>|y-z|>|x-z|/2\}.
\end{split}
\]
Hence,
\[
\begin{split}
I= \sum_{i=1}^4\int_{U_i} \left(\frac{1}{|x-y|^{n-2\sigma}}-\frac{1}{|z-y|^{n-2\sigma}}\right)V(y)u(y)\,\ud y:=I_1+I_2+I_3+I_4.
\end{split}
\]
For $I_1$, we have
\[
\begin{split}
|I_1|&\le \int_{U_1} \left(\frac{1}{|x-y|^{n-2\sigma}}\right)V(y)u(y)\,\ud y\\
&\le \|u\|_{L^\infty(B_{1.9})}\|V\|_{L^p(U_1)}\||x-y|^{2\sigma-n}\|_{L^s}|U_1|^{1/t}\\
&\le C\|u\|_{L^\infty(B_{1.9})}\|V\|_{L^p(B_3)}|x-z|^{\frac{n}{t}}
\end{split}
\]
for some fixed $1<s<\frac{n}{n-2\sigma}$ and $t>1$ satisfying $\frac{1}{p}+\frac{1}{s}+\frac{1}{t}=1$, where in the second inequality we used H\"older inequality. Similarly,
\[
|I_2|\le C\|u\|_{L^\infty(B_{1.9})}\|V\|_{L^p(B_3)}|x-z|^{\frac{n}{t}}.
\]
For $I_3$, we have
\[
\begin{split}
|I_3|&\le C\int_{U_3} \left(\frac{|x-z|}{|x-y|^{n-2\sigma+1}}\right)V(y)u(y)\,\ud y\\
&\le C|x-z|^\beta\int_{U_3} \left(\frac{1}{|x-y|^{n-2\sigma+\beta}}\right)V(y)u(y)\,\ud y\\
&\le C|x-z|^\beta\|u\|_{L^\infty(B_{1.9})}\|V\|_{L^p(U_3)}\||x-y|^{2\sigma-n-\beta}\|_{L^{\bar s}}|U_3|^{1/\bar t}\\
&\le C\|u\|_{L^\infty(B_{1.9})}\|V\|_{L^p(B_3)}|x-z|^{\beta}
\end{split}
\]
for some $\beta>0$ small, $1<\bar s<\frac{n}{n-2\sigma+\beta}$ and $\bar t>1$, since $p>\frac{n}{2\sigma}$. Similarly,
\[
|I_4|\le C\|u\|_{L^\infty(B_{1.9})}\|V\|_{L^p(B_3)}|x-z|^{\beta}.
\]
Thus, the H\"older estimate of $u$ follows.
\end{proof}

Let us consider nonnegative solutions of the integral equation
\be \label{eq:lineareq1}
u(x)=\int_{\R^n}\frac{V(y)u(y)}{|x-y|^{n-2\sigma}}\,\ud y \quad  a.e  \mbox{ in } B_3,
\ee
where $0<\sigma<n/2$.

\begin{thm} \label{thm:harnack} For $n\ge 1, 0<\sigma<n/2$, $r>n/(n-2\sigma)$ and $p>n/2\sigma$, let $0\le V\in L^p(B_3)$, $0\le u  \in  L^r(B_3)$ and $0\le Vu\in L^1_{loc}(\R^n)$. If  $u$ satisfies \eqref{eq:lineareq1},  then $u\in C^\al(B_1)$,
\be\label{eq:holder estimate for linear equation}
\|u\|_{C^\al(B_1)} \leq C \|u\|_{L^r(B_3)},
\ee
and $u$ satisfies the Harnack inequality
\be \label{eq:linear HI}
\max_{\bar B_{1}}u\leq C\min_{\bar B_{1}}u,
\ee
where $C>0$ and $\al\in (0,1)$ depend only on $n,\sigma, p, $ and an upper bound of $ \|V\|_{L^p(B_3)}$.
\end{thm}

\begin{proof} 
%The assumption implies
%\[
%\int_{B_\va^c} \frac{V(y)u(y)}{|y|^{n-2\sigma}}\,\ud y<\infty \quad \forall \va>0,
%\]
%where $B_\va^c=\R^n\setminus B_\va$ denotes the complement of $B_\va$. 
For $x\in B_2$, set
\[
h(x)= \int_{B_3^c} \frac{V(y)u(y)}{|x-y|^{n-2\sigma}}\,\ud y.
\]
Then
\[
u(x)= \int_{B_3} \frac{V(y)u(y)}{|x-y|^{n-2\sigma}}\,\ud y+ h(x).
\]
Let $x_0$ be a point $\bar B_1$ such that
\[
u(x_0)\leq \frac{1}{|B_1|}\|u\|_{L^1(B_1)}.
\]
Since $ uV\ge 0$ in $\R^n$, we have, for $x\in B_2$,
\[
h(x)= \int_{B_3^c} \frac{|x_0-y|^{n-2\sigma}}{|x-y|^{n-2\sigma}} \frac{V(y)u(y)}{|x_0-y|^{n-2\sigma}}\,\ud y \leq 4^{n-2\sigma} h(x_0) \leq 4^{n-2\sigma} u(x_0)\le \frac{4^{n-2\sigma}}{|B_1|}\|u\|_{L^1(B_1)}.
\]
By the same argument, we can show that 
\[
\max_{\bar B_{1}}h\le 4^{n-2\sigma} \min_{\bar B_1} h,
\]
and for every $k=1,2,\dots $
\be \label{eq:remain term regu}
\|h\|_{C^k(B_2)} \leq C(k) \|u\|_{L^1(B_3)}.
\ee
Hence, this theorem follows from Proposition \ref{prop:har} and Proposition \ref{prop:holder}.
\end{proof}

We also have Schauder type estimates.

\begin{thm}\label{thm:schauder} In addition to the assumptions in Theorem \ref{thm:harnack}, we assume that $V\in C^\al(B_3)$ for some $\al>0$ but not an integer, then $u\in C^{2\sigma+\al'}(B_1)$ and
\[
\|u\|_{C^{2\sigma+\al'}(B_1)}\leq C  \|u\|_{L^r(B_3)},
\]
where $\al'=\al$ if $2\sigma+\al\notin \mathbb{N}^+$, otherwise $\al'$ can be any positive constant less than $\al$. Here $C>0$ depends only on $n,\sigma, \al$ and an upper bound of $\|V\|_{C^\al(B_3)}$.
\end{thm}

\begin{proof}
Let $\eta$ be a nonnegative smooth cut-off function supported in $B_{2.5}$ such that $\eta\equiv 1$ in $B_2$. Then
\[
u(x)=\int_{\R^n}\frac{\eta(y)V(y)u(y)}{|x-y|^{n-2\sigma}}\,\ud y+\int_{\R^n}\frac{(1-\eta(y))V(y)u(y)}{|x-y|^{n-2\sigma}}\,\ud y:=I+II \quad  a.e  \mbox{ in } B_3,
\]
It is clear that
\[
\|II\|_{C^{2\sigma+\al'}(B_1)}\leq C  \|u\|_{L^\infty(B_{2.5})}\le C  \|u\|_{L^r(B_3)}.
\]
It follows from \eqref{eq:holder estimate for linear equation} and standard Riesz potential theory that
\[
\|I\|_{C^{2\sigma+\beta}(B_1)}\leq C  \|\eta V u\|_{C^\beta(\R^n)}\le C\|u\|_{L^r(B_3)}
\]
for some $\beta>0$. Then the theorem follows from bootstrap arguments.
\end{proof}

\subsection{Blow up profiles of solutions of nonlinear integral equations}

Let $\{\tau_i\}_{i=1}^\infty$ be a sequence of nonnegative constants satisfying $\lim_{i\to \infty}\tau_i=0$, and set
\[p_i=(n+2\sigma)/(n-2\sigma)-\tau_i.
\]
Suppose that $0\le u_i\in L^{\infty}_{loc}(\R^n)$ satisfies the nonlinear integral equation
\be \label{eq:main}
u_i(x)= \int_{\R^n} \frac{K_i(y)u_i(y)^{p_i}}{|x-y|^{n-2\sigma} }\,\ud y \quad \mbox{in }\om,
\ee where $\om$ is a domain in $\R^n$ and $K_i$ are nonnegative bounded functions in $\R^n$. We assume that $K_i\in C^{1}(\om)$ and, for some positive constants $A_1$ and $A_2$,
 \be\label{eq:Kca}
\begin{split}
  1/A_1\leq K_i(x)&\leq A_1 \quad \mbox{for all } x\in \om,\\
  \|K_i\|_{C^{1}(\om)}&\leq A_2,\\
  \|K_i\|_{C^{1,1} (\om)}& \le A_2,  \quad \mbox{if }\sigma \le \frac{1}{2}.
\end{split}
 \ee
 It follows from Theorem \ref{thm:schauder} that  for any $\al \in (0,1)$ and  smooth $\om'\subset \subset \om$
 \be \begin{split} \label{eq:reg}
 u_i\in C^{2\sigma+\al}(\om'), & \quad  \mbox{if }\sigma>\frac12; \\
 u_i\in C^{2+\al}(\om'), & \quad \mbox{if }\sigma= \frac12;\\
u_i\in C^{2+2\sigma}(\om') , &\quad \mbox{if }\sigma <\frac12.
 \end{split}
 \ee

We say that $\{u_i\}$ blows up if $\|u_i\|_{L^\infty(\om)}\to \infty$ as $i\to \infty$.

\begin{defn}\label{def4.1}
Suppose that $\{K_i\}$ satisfies \eqref{eq:Kca} and $\{u_i\}$ satisfies \eqref{eq:main}.
We say a point $\overline x\in \om$ is an isolated blow up point of $\{u_i\}$ if there exist
$0<\overline r<\mbox{dist}(\overline x,\om)$, $\overline C>0$, and a sequence $x_i$ tending to $\overline x$, such that,
$x_i$ is a local maximum of $u_i$, $u_i(x_i)\rightarrow \infty$ and
\[
u_i(x)\leq \overline C |x-x_i|^{-2\sigma/(p_i-1)} \quad \mbox{for all } x\in B_{\overline r}(x_i).
\]
\end{defn}

Without loss of generality, we may assume that $B_3\subset \om$.

Let $x_i\rightarrow \overline x$ be an isolated blow up of $u_i$. Define
\be\label{def:average}
\overline u_i(r)=\frac{1}{|\pa B_r|} \int_{\pa B_r(x_i)}u_i,\quad r>0,
\ee
and
\[
\overline w_j(r)=r^{2\sigma/(p_i-1)}\overline u_i(r), \quad r>0.
\]

\begin{defn}\label{def4.2}
We say $x_i \to \overline x\in \om$ is an isolated simple blow up point, if $x_i \to \overline x$ is an isolated blow up point, such that, for some
$\rho>0$ (independent of $i$) $\overline w_i$ has precisely one critical point in $(0,\rho)$ for large $i$.
\end{defn}

If $x_i\to 0$ is an isolated blow up point, then we will have the following Harnack inequality in the annulus centered at $0$.

\begin{lem}\label{lem:harnack} Suppose that $0\le u_i\in L^{\infty}_{loc}(\R^n)$ satisfies \eqref{eq:main} with $K_i$ satisfying \eqref{eq:Kca}.  Suppose that $x_i\rightarrow 0$
is an isolated blow up point of $\{u_i\}$, i.e., for some positive constants $A_3$ and $\bar r$ independent of $i$,
\be\label{4.7}
|x-x_i|^{2\sigma/(p_i-1)}u_i(x)\leq A_3\quad \mbox{for all } x\in B_{\bar r}\subset\om.
\ee
Then for any $0<r<\frac13 \overline r$, we have the following Harnack inequality
\[
\sup_{B_{2r}(x_i)\setminus\overline{B_{r/2}(x_i)}} u_i\leq C \inf_{B_{2r}(x_i)\setminus\overline{B_{r/2}(x_i)}} u_i,
\]
where $C$ is a positive constant depending only on $n, \sigma, A_3, \bar r$ and $\dsup_i\|K_i\|_{L^\infty(B_{\overline r}(x_i))}$.
\end{lem}

\begin{proof}
For $0<r<\frac{\bar r}{3}$, set
\[
w_i(x)=r^{2\sigma/(p_i-1)}u_i(x_i+rx)\quad \mbox{for } x\in B_3.
\]
We see that
\[
w_i(x)= \int_{\R^n} \frac{K_i(x_i+ry)w_i(y)^{p_i}}{|x-y|^{n-2\sigma}}\,\ud y \quad x\in B_3.
\]
Since $x_i\to 0$ is an isolated blow up point of $u_i$,
\[
w_i(x)\leq A_3 |x|^{-2\sigma/(p_i-1)}\quad \mbox{for all } x\in B_3.
\]
Set $V_i(y)=K_i(x_i+ry)w_i(y)^{p_i-1}$. Clearly, $\|V_i\|_{L^\infty(B_{5/2}\setminus  B_{1/4})} \le A_3 4^{2\sigma/(p_i-1)}$.  Covering $B_{5/2}\setminus  B_{1/4}$ by finite many balls, then
Lemma \ref{lem:harnack} follows immediately from the Harnack inequality in Theorem \ref{thm:harnack}.
\end{proof}

The next proposition shows that if $x_i\to 0$ is an isolated blow up point, then the blow up profile has to be a standard bubble.

\begin{prop}\label{prop:blow up a bubble}
Assume the assumptions in Lemma \ref{lem:harnack}.
Then for any $R_i\rightarrow \infty$, $\va_i\rightarrow 0^+$, we have,
after passing to a subsequence (still denoted as $\{u_i\}$,
$\{x_i\}$, etc. ...), that
\[
\|m_i^{-1}u_i(m_i^{-(p_i-1)/2\sigma}\cdot+x_i)-(1+k_i|\cdot|^2)^{(2\sigma-n)/2}\|_{C^2(B_{2R_i}(0))}\leq \va_i,
\]
\[
r_i:=R_im_i^{-(p_i-1)/2\sigma}\rightarrow 0\quad \mbox{as}\quad i\rightarrow \infty,
\]
where $m_i=u_i(x_i)$ and $k_i=\Big(\frac{K_i(x_i)\pi^{n/2}\Gamma(\sigma)}{\Gamma(\frac n2 +\sigma)}\Big)^{1/\sigma}$.
\end{prop}

\begin{proof} Let
\[
\varphi_i(x)=m_i^{-1} u_i(m_i^{-(p_i-1)/2\sigma} x+x_i) \quad \mbox{for }x\in \R^n.
\]
It is easy to see that for $|x|<\overline r m_i^{(p_i-1)/2\sigma}$ we have
\be \label{eq:scal}
\varphi_i(x)= \int_{\R^n} \frac{ K_i(m^{-(p_i-1)/2\sigma} y+x_i) \varphi_i(y)^{p_i}}{|x-y|^{n-2\sigma}}\,\ud y,
\ee
\be \label{eq:scalbound}
\varphi_i(0)=1, \quad \nabla \varphi_i(0)=0, \quad  0<\varphi_i(x)\leq A_3 |x|^{-2\sigma/(p_i-1)}.
\ee
For any $R>0$, we claim that
\be \label{eq:scalbound2}
\|\varphi_i\|_{C^{2,\al}(B_R)}\leq C(R)
\ee
for some $\al\in (0,1)$ and all sufficiently large $i$. Indeed, by \eqref{eq:scalbound} and Theorem \ref{thm:schauder}, it suffices to prove that $\varphi_i \le C$ in $B_1$.
If $\varphi_i(\bar x_i)=\sup_{B_{1}} \varphi_i \to \infty $, set
\[
\tilde \varphi_i(z)=\varphi_i(\bar x_i)^{-1}\varphi_i(\varphi_i(\bar x_i)^{-(p_i-1)/2\sigma}z+\bar x_i)\leq 1\quad \mbox{for }|z| \le \frac12 \varphi_i(\bar x_i)^{(p_i-1)/2\sigma}.
\]
By \eqref{eq:scalbound},
\[
\tilde \varphi_i(z_i)= \varphi_i(\bar x_i)^{-1} \varphi_i(0)\to 0
\]  for $z_i=-\varphi_i(\bar x_i)^{(p_i-1)/2\sigma}\bar x_i$.  Since $\varphi_i(\bar x_i) \leq  A_3 |\bar x_i|^{-2\sigma/(p_i-1)}$, we have $|z_i|\leq A_3^{2\sigma/(p_1-1)}$. Hence, we can find $t>0$ independent of $i$ such that such that $z_i\in B_t$.
 Applying Theorem \ref{thm:harnack} to $\tilde \varphi_i$ in $B_{2t}$ (since $\tilde\varphi_i$ satisfies a similar equation to \eqref{eq:scal}), we have
\[
1=\tilde \varphi_i(0)\le C \tilde \varphi_i(z_i)\to 0,
\]
which is impossible. Thus \eqref{eq:scalbound2} is valid.

It follows from \eqref{eq:scalbound2} that,  after passing to a subsequence if necessary,  $\varphi_i \to \varphi$ in $C^2_{loc}(\R^n)$ for some $\varphi\in C^2(\R^n)$. We write \eqref{eq:scal} as
\[
\varphi_i(x)= \int_{B_R} \frac{K_i(m_i^{-(p_i-1)/2\sigma} y+x_i) \varphi_i(y)^{p_i}}{|x-y|^{n-2\sigma}}\,\ud y+h_i(R,x),
\]
where
\[
h_i(R,x):=\int_{B_R^c} \frac{K_i(m_i^{-(p_i-1)/2\sigma} y+x_i) \varphi_i(y)^{p_i}}{|x-y|^{n-2\sigma}}\,\ud y.
\]
Since  $K_i$ and $\varphi_i$ are nonnegative, by \eqref{eq:scalbound2} we have $\|h_i(R,\cdot)\|_{C^1(B_{R-1})}\le C(R)$ for all $i$ large. Thus, after passing to a subsequence, $h_i(R,x) \to h(R,x)$ in $C^{1/2}(B_{R-1})$.
Therefore,
\be \label{eq:limit1}
h(R,x)=\varphi(x)-\int_{B_R}\frac{K_0 \varphi(y)^{\frac{n+2\sigma}{n-2\sigma}}}{|x-y|^{n-2\sigma}}\,\ud y
\ee
for $x\in B_{R-1}$, where $K_0=\dlim_{i\to\infty}K_i(x_i)$. Moreover,  $h(R,x)\ge 0$, and is non-increasing in $R$. Note that for $R>>|x|$,
\begin{align*}
\frac{R^{n-2\sigma}}{(R+|x|)^{n-2\sigma}} h_i(R, 0)& \leq h_i(R, x)\\&
=\int_{B_R^c}\frac{K_i(m_i^{-(p_i-1)/2\sigma} y+x_i) \varphi_i(y)^{p_i}}{|y|^{n-2\sigma}} \frac{|y|^{n-2\sigma}}{|x-y|^{n-2\sigma}}\,\ud y\\
&\leq \frac{R^{n-2\sigma}}{(R-|x|)^{n-2\sigma}} h_i(R, 0).
\end{align*}
Hence,
\be \label{eq:remainder1}
\lim_{R\to \infty}h(R,x)= \lim_{R\to \infty} h(R,0).
\ee
Sending $R$ to $\infty$ in \eqref{eq:limit1},  it follows from Lebesgue's monotone convergence theorem that
\[
\varphi(x)= \int_{\R^n}\frac{K_0\varphi(y)^{\frac{n+2\sigma}{n-2\sigma}}}{|x-y|^{n-2\sigma}} \ud y+c_0 \quad x\in \R^n,
\]
where $ c_0=\lim_{R\to \infty}h(R,0)\geq 0$. We claim that $c_0=0$. If not,
\[
\varphi(x)-c_0= \int_{\R^n}\frac{K_0\varphi(y)^{\frac{n+2\sigma}{n-2\sigma}}}{|y|^{n-2\sigma}}\,\ud y >0,
\]
which implies that
\[
1=\varphi(0)\geq \int_{\R^n}\frac{K_0c_0^{\frac{n+2\sigma}{n-2\sigma}}}{|x-y|^{n-2\sigma}}=\infty.
\]
This is impossible. It follows the classification theorem in \cite{CLO} or \cite{Li04} that
\[
\varphi(x)=\left(1+\lim_{i\to\infty}k_i|x|^2\right)^{-\frac{n-2\sigma}{2}}
\]
with $k_i=\Big(\frac{K_i(x_i)\pi^{n/2}\Gamma(\sigma)}{\Gamma(\frac n2 +\sigma)}\Big)^{1/\sigma}$, where we used that $\varphi(0)=1$ and $\nabla \varphi(0)=0$.
\end{proof}
Note that since passing to subsequences does not affect our proofs, in the rest of the paper we will always choose $R_i\to\infty$ first, and then $\va_i\to 0^+$ as small as we wish (depending on $R_i$) and then choose our subsequence $\{u_i\}$ to work with.

\begin{rem}
For $\sigma\ge 1$, we can have a simpler proof of local boundedness of $\varphi_i$. Indeed, Since
\[
\Delta_x \frac{1}{|x-z|^{n-2\sigma}}= \frac{2(n-2\sigma)(1-\sigma)}{|x-z|^{n-2\sigma+2}} \quad \mbox{if }x\neq z,
\]
$\Delta w_i\le 0$ if $\sigma\ge 1$. By maximum principle, we have $\min_{\pa B_R} \varphi_i\leq \varphi_i(0)=1$ for any fixed $R>0$. Then the local boundedness of $\varphi_i$ follows from Lemma \ref{lem:harnack}.
\end{rem}

Consequently, we have the following lower bound of $u_i$ near isolated blow up points. To save the notation, we may assume that $\bar r=2$. 

\begin{prop}\label{prop:lower bounded by bubble} Under the hypotheses of Proposition \ref{prop:blow up a bubble}, there exists some positive constant $C=C(n,\sigma, A_1,A_2,A_3)$ such that,
\[
u_i(x)\geq C^{-1}m_i(1+k_im_i^{(p_i-1)/\sigma}|x-x_i|^2)^{(2\sigma-n)/2} \quad\mbox{for all } |x-x_i|\leq 1.
\]
In particular, for any $e\in \mathbb{R}^n$, $|e|=1$, we have
\[
u_i(x_i+e)\geq C^{-1}m_i^{-1+((n-2\sigma)/2\sigma)\tau_i},
\]
where $\tau_i=(n+2\sigma)/(n-2\sigma)-p_i$.
\end{prop}
\begin{proof}
By change of variables and using Proposition \ref{prop:blow up a bubble}, we have for $r_i\le |x-x_i|\le 1$,
\be\label{eq:lower bound bubble}
\begin{split}
u(x)&\ge C\int_{|y-x_i|\le r_i} \frac{u_{i}(y)^{p_i}}{|x-y|^{n-2\sigma}} \,\ud y\\&
 \ge Cm_i\int_{|z|\le R_i} \frac{\big(m_i^{-1}u_{i}(m_i^{-(p_i-1)/2\sigma}z+x_i)\big)^{p_i}}{|m_i^{(p_i-1)/2\sigma}(x-x_i)-z|^{n-2\sigma}} \,\ud z
 \\&
 \ge \frac 12 C m_i \int_{|z|\le R_i} \frac{U_i(z)^{p_i}}{|m_i^{(p_i-1)/2\sigma}(x-x_i)-z|^{n-2\sigma}} \,\ud z\\&
  \ge \frac 12 C m_i \int_{|z|\le R_i} \frac{U_i(z)^{\frac{n+2\sigma}{n-2\sigma}}}{|m_i^{(p_i-1)/2\sigma}(x-x_i)-z|^{n-2\sigma}} \,\ud z\\&
 \ge \frac 14 C m_i \int_{\R^n} \frac{U_i(z)^{\frac{n+2\sigma}{n-2\sigma}}}{|m_i^{(p_i-1)/2\sigma}(x-x_i)-z|^{n-2\sigma}} \,\ud z\\&
 =\frac 14 Cm_i  U_i(m_i^{(p_i-1)/2\sigma}(x-x_i)),
 \end{split}
\ee
where
\[
U_i(z)=\left(\frac{1}{1+k_i|z|^2}\right)^{(n-2\sigma)/2}.
\]
The proposition follows immediately from the above and Proposition \ref{prop:blow up a bubble}.
\end{proof}

To obtain the desired upper bound of $u_i$ near isolated simple blow up points, we need an auxiliary bound in the below, and a Pohozaev type identity in Proposition \ref{prop:pohozaev}. Pohozaev type identities for fractional Laplacian and some applications can be found in \cite{RS1,RS2}.

\begin{lem} \label{lem:upbound1}  Under the hypotheses of Proposition \ref{prop:blow up a bubble} with $\bar r=2$,
and in addition that $x_i\to 0$ is also an isolated simple blow up point with the constant $\rho$, then there exists $\delta_i>0$, $\delta_i=O(R_i^{-2\sigma+o(1)})$,
such that
\[
u_i(x)\leq C_1 R_i^{(n-2\sigma)\tau_i}u_i(x_i)^{-\lda_i}|x-x_i|^{2\sigma-n+\delta_i}\quad \mbox{for all }r_i\leq |x-x_i|\leq 1,
\]
where $\lda_i=(n-2\sigma-\delta_i)(p_i-1)/2\sigma-1$ and $C_1$ is some positive constant depending only on $n,\sigma, A_1,A_3$ and $\rho$.

\end{lem}

\begin{proof}  We divide the proof into several steps.

Step 1. From Proposition \ref{prop:blow up a bubble}, we see that
\begin{align}
u_i(x)&\le C m_i \left(\frac{1}{1+|m_i^{(p_i-1)/2\sigma}(x-x_i)|^2}\right)^{\frac{n-2\sigma}{2}} \nonumber \\ &
\le C u_i(x_i)R_i^{2\sigma-n} \quad \mbox{for all } |x-x_i|=r_i=R_i m_i^{-(p_i-1)/2\sigma}.
                             \label{4.8}
\end{align}
Let $\overline u_i(r)$ be the average of $u_i$ over the sphere of radius $r$ centered at $x_i$.
It follows from the assumption of isolated simple
blow up points and Proposition \ref{prop:blow up a bubble} that
\be\label{4.9}
r^{2\sigma/(p_i-1)}\overline u_i(r) \quad \mbox{is strictly decreasing for $r_i<r<\rho$}.
\ee
By Lemma \ref{lem:harnack},  \eqref{4.9} and \eqref{4.8},  we have, for all $r_i<|x-x_i|<\rho$,
\[
\begin{split}
|x-x_i|^{2\sigma/(p_i-1)}u_i(x)&\leq C|x-x_i|^{2\sigma/(p_i-1)}\overline u_i(|x-x_i|)\\&
\leq C r_i^{2\sigma/(p_i-1)}\overline u_i(r_i)\\&
\leq CR_i^{\frac{2\sigma-n}{2}+o(1)},
\end{split}
\]
where $o(1)$ denotes some quantity tending to $0$ as $i\to \infty$. Thus,
\be \label{eq:coeff}
u_i(x)^{p_i-1}\leq C R_i^{-2\sigma+o(1)}|x-x_i|^{-2\sigma} \quad \mbox{for all } r_i\leq |x-x_i|\le \rho.
\ee

\medskip

Step 2.  Let
\[
\mathcal{L}_i\phi(y):= \int_{\R^n} \frac{K_i(z) u_i(z)^{p_i-1}\phi(z)}{|y-z|^{n-2\sigma}}\,\ud z.
\]
Thus
\[
u_i=\mathcal{L}_i u_i.
\]
Note that for $2\sigma<\mu<n$ and $0<|x|<2$,
\begin{align*}
\int_{\R^n} \frac{1}{|x-y|^{n-2\sigma}|y|^{\mu}}\,\ud y&= |x|^{2\sigma-n} \int_{\R^n}
\frac{1}{||x|^{-1}x-|x|^{-1}y|^{n-2\sigma}|y|^{\mu}}\,\ud y \\&
= |x|^{-\mu+2\sigma} \int_{\R^n} \frac{1}{||x|^{-1}x-z|^{n-2\sigma}|z|^{\mu}}\,\ud z \\&
 \le C\Big( \frac{1}{n-\mu}+\frac{1}{\mu- 2\sigma} +1\Big)|x|^{-\mu+2\sigma},
\end{align*}
where we did the change of variables $y=|x|z$.
By \eqref{eq:coeff}, one can properly choose $0<\delta_i=O(R_i^{-2\sigma+o(1)})$ such that
\be \label{eq:a1}
\int_{r_i<|y-x_i|<\rho} \frac{K_i(y)u_i(y)^{p_i-1} |y-x_i|^{-\delta_i} }{|x-y|^{n-2\sigma}}\,\ud y\leq \frac{1}{4} |x-x_i|^{-\delta_i},
\ee
and
\be \label{eq:a2}
\int_{r_i<|y-x_i|<\rho} \frac{K_i(y)u_i(y)^{p_i-1} |y-x_i|^{2\sigma-n+\delta_i} }{|x-y|^{n-2\sigma}}\,\ud y\leq \frac{1}{4} |x-x_i|^{2\sigma-n+\delta_i},
\ee
for all $r_i< |x-x_i|<\rho$.

Set $M_i:=4\cdot  2^{n-2\sigma}\max_{\pa B_\rho(x_i)} u_i$,
\[
f_i(x):=M_i \rho^{\delta_i} |x-x_i|^{-\delta_i}+A R_i^{(n-2\sigma) \tau_i} m_i^{-\lda_i} |x-x_i|^{2\sigma-n+\delta_i},
\]  and
\[
\phi_i(x)=\begin{cases}
f_i(x), & \quad r_i< |x-x_i|< \rho,\\
u_i(x),&\quad \mbox{otherwise} ,
\end{cases}
\]
where $A>1$ will be chosen later.

By \eqref{eq:a1} and \eqref{eq:a2}, we have for $r_i<|x-x_i|< \rho$.
\begin{align*}
\mathcal{L}_i \phi_i (x)&=  \int_{\R^n} \frac{K_i(y)u_i(y)^{p_i-1} \phi_i(y)}{|x-y|^{n-2\sigma}}\,\ud y\\&
=\int_{|y-x_i|\le r_i} + \int_{r_i<|y-x_i|<\rho} +  \int_{\rho\le |y-x_i|} \frac{K_i(y)u_i(y)^{p_i-1} \phi_i(y)}{|x-y|^{n-2\sigma}}\,\ud y
\\&
 \leq A_1 \int_{|y-x_i|\le r_i} \frac{u_{i}(y)^{p_i}}{|x-y|^{n-2\sigma}} \,\ud y+ \frac{f_i}{4}+   \int_{\rho\le |y-x_i|} \frac{K_i(y)u_i(y)^{p_i-1} \phi_i(y)}{|x-y|^{n-2\sigma}}\,\ud y.
\end{align*}
To estimate the third term, we let $\bar x=\rho\frac{x-x_i}{|x-x_i|}\in\partial B_{\rho}(x_i)$, and then
\be\label{eq:trick sub}
\begin{split}
&\int_{\rho\le |y-x_i|} \frac{K_i(y)u_i(y)^{p_i-1} \phi_i(y)}{|x-y|^{n-2\sigma}}\,\ud y\\
&=\int_{\rho\le |y-x_i|} \frac{|\bar x-y|^{n-2\sigma}}{| x-y|^{n-2\sigma}}\frac{K_i(y)u_i(y)^{p_i-1} \phi_i(y)}{|\bar x-y|^{n-2\sigma}}\,\ud y\\
&\le 2^{n-2\sigma} \int_{\rho\le |y-x_i|} \frac{K_i(y)u_i(y)^{p_i-1} \phi_i(y)}{|\bar x-y|^{n-2\sigma}}\,\ud y\\
&\le 2^{n-2\sigma} u_i(\bar x)\le 2^{n-2\sigma} \max_{\partial B_{\rho}(x_i)}u_i\le M_i/4.
\end{split}
\ee
To estimate the first term, we use change of variables, Proposition \ref{prop:blow up a bubble} and the computations in \eqref{eq:lower bound bubble} that,
\begin{align*}
\int_{|y-x_i|\le r_i} \frac{u_{i}(y)^{p_i}}{|x-y|^{n-2\sigma}} \,\ud y&
 =m_i\int_{|z|\le R_i} \frac{\big(m_i^{-1}u_{i}(m_i^{-(p_i-1)/2\sigma}z+x_i)\big)^{p_i}}{|m_i^{(p_i-1)/2\sigma}(x-x_i)-z|^{n-2\sigma}} \,\ud z
 \\&
 \le 2m_i \int_{|z|\le R_i} \frac{U_i(z)^{p_i}}{|m_i^{(p_i-1)/2\sigma}(x-x_i)-z|^{n-2\sigma}} \,\ud z\\&
 \le C m_i R_i^{(n-2\sigma)\tau_i} \int_{\R^n} \frac{U_i(z)^{\frac{n+2\sigma}{n-2\sigma}}}{|m_i^{(p_i-1)/2\sigma}(x-x_i)-z|^{n-2\sigma}} \,\ud z\\&
 =Cm_i R_i^{(n-2\sigma)\tau_i} U_i(m_i^{(p_i-1)/2\sigma}(x-x_i)),
\end{align*}
where
\[
U_i(z)=\left(\frac{1}{1+k_i|z|^2}\right)^{(n-2\sigma)/2}.
\]
Since $|x-x_i|>r_i$,  we see
\begin{align*}
m_i U_i(m_i^{(p_i-1)/2\sigma}(x-x_i)) &\le Cm_i^{1-(p_i-1)(n-2\sigma)/2\sigma}|x-x_i|^{2\sigma-n} \\&
\le Cm_i^{-\lda_i}|x-x_i|^{2\sigma-n+\delta_i}.
\end{align*}
Therefore, we conclude that
\be \label{eq:intineq}
\mathcal{L}_i \phi_i(x) \leq \phi_i(x)\quad \mbox{for all } r_i\le |x-x_i| \leq \rho,
\ee
provided that $A$ is sufficiently large and independent of $i$.

\medskip

Step 3. In view of \eqref{4.8}, we may choose $A$ large such that $f_i \geq u_i $ on $\pa B_{r_i}(x_i)$.  By the choice of $M_i$, we know that $f_i \geq u_i $ on $\pa B_{\rho}(x_i)$. We claim that \[u_i\le \phi_i\quad\mbox{in }\R^n.\]  Indeed, if not, let
\[
1<t_i:=\inf\{t>1: t\phi_i\ge u_i\quad\mbox{in }\R^n\}<\infty.
\] Since $t_i\phi_i\ge u_i $ in $B_{r_i}(x_i)\cup  B^c_\rho(x_i)$, by the continuity there exists $y_i\in B_\rho(x_i)\setminus \bar B_{r_i}(x_i)$ such that
\[
0=t_i\phi_i(y_i)-u_i(y_i)\ge \mathcal{L}_i(t_i\phi_i-u_i)(y_i)>0.
\]
This is a contradiction. Therefore, the claim is proved.

\medskip

Step 4.   For $r_i<\theta<\rho$,
\[
\begin{split}
\rho^{2\sigma/(p_i-1)}M_i&\leq C \rho^{2\sigma/(p_i-1)}\overline u_i(\rho)\\
&\leq C\theta^{2\sigma/(p_i-1)}\overline u_i(\theta)\\
&\leq C\theta^{2\sigma/(p_i-1)}\{M_i\rho^{\delta_i}\theta^{-\delta_i}+AR_i^{(n-2\sigma)\tau_i}m_i^{-\lda_i}\theta^{2\sigma-n+\delta_i}\}.
\end{split}
\]
Choose $\theta=\theta(n,\sigma,\rho,A_1,A_2, A_3)$ sufficiently small so that
\[
C\theta^{2\sigma/(p_i-1)}\rho^{\delta_i}\theta^{-\delta_i}\leq \frac12 \rho^{2\sigma/(p_i-1)}.
\]
It follows that
\[
M_i\leq CR_i^{(n-2\sigma)\tau_i} m_i^{-\lda_i}.
\]
The lemma follows from the above and the Harnack inequality in Lemma \ref{lem:harnack}.
\end{proof}

\begin{prop}[Pohozaev type identity] \label{prop:pohozaev} Let $u\ge 0$ in $\R^n$, and $u\in C(\overline B_R)$ be a solution of
\[
u(x)= \int_{B_R} \frac{K(y)u(y)^{p}}{|x-y|^{n-2\sigma}}\,\ud y+ h_R(x),
\]
where $1<p\le \frac{n+2\sigma}{n-2\sigma}$, and $h_R(x)\in C^1(B_R)$, $\nabla h_R\in L^1(B_R)$.
Then
\begin{align*}
&\left(\frac{n-2\sigma}{2}-\frac{n}{p+1}\right) \int_{B_R} K(x)u(x)^{p+1}\,\ud x-\frac{1}{p+1} \int_{B_R} x\nabla K(x) u(x)^{p+1}\,\ud x \\ &
=\frac{n-2\sigma}{2} \int_{B_R} K(x) u(x)^p h_R(x)\,\ud x+ \int_{B_R} x\nabla h_R(x) K(x)u(x)^p \,\ud x
\\& \quad - \frac{R}{p+1} \int_{\pa B_R} K(x) u(x)^{p+1}\,\ud s.
\end{align*}
\end{prop}

\begin{proof} We first prove the case $\sigma>1/2$. Note that
\begin{align*}
&\frac{1}{p+1}\int_{B_R} x K(x) \nabla u(x)^{p+1} \,\ud x \\&=\int_{B_R} xK(x) u(x)^p \nabla u(x)\\&
= (2\sigma-n) \int_{B_R} xK(x) u(x)^p \int_{B_R} \frac{(x-y)K(y)u(y)^{p}}{|x-y|^{n+2-2\sigma}}\,\ud y\ud x \\&\quad +\int_{B_R} xK(x) u(x)^p \nabla h_R(x)\,\ud x.
\end{align*}
By the divergence theorem,
\[
\int_{B_R} x K(x) \nabla u(x)^{p+1}\,\ud x= - \int_{B_R} (nK(x)+x\nabla K(x)) u(x)^{p+1}\,\ud x+R\int_{\pa B_R} K(x) u(x)^{p+1}\,\ud s.
\]
By direct computations,
\begin{align*}
&\int_{B_R} xK(x) u(x)^p \int_{B_R} \frac{(x-y)K(y)u(y)^{p}}{|x-y|^{n+2-2\sigma}}\,\ud y\ud x\\&
=\frac12 \int_{B_R} K(x) u(x)^p \int_{B_R} \frac{(|x-y|^2+(|x|^2-|y|^2))K(y)u(y)^{p}}{|x-y|^{n+2-2\sigma}}\,\ud y\ud x\\&
=\frac12 \int_{B_R} K(x) u(x)^p \int_{B_R} \frac{K(y)u(y)^{p}}{|x-y|^{n-2\sigma}}\,\ud y\ud x\\&
=\frac12 \int_{B_R} K(x) u(x)^p(u(x)-h_R(x))\,\ud x.
\end{align*}
Therefore, the proposition follows immediately for $\sigma>1/2$.

When $0<\sigma\le 1/2$, we can truncate the kernel $\frac{1}{|x-y|^{n-2\sigma}}$ by
\[
G_\va (x,y)=  \begin{cases} \frac{2\sigma-n}{2}\va^{-n+2\sigma-2}|x-y|^2+\frac{n-2\sigma+2}{2}\va^{-n+2\sigma}, & \quad\mbox{if } |x-y|\le \va\\
|x-y|^{-n+2\sigma}, &\quad\mbox{if } |x-y|> \va,
  \end{cases}
\]
so that $G_\va\in C^1$, where $0<\va<R/2$ is small, and define
\[
v_\va=\int_{B_R} G_\va(x,y)K(y)u(y)^{p}\,\ud y+ h_R(x).
\]
Then the identity follows from an approximation argument by doing the same computations as above for $v_\va$ and sending $\va\to 0$ in the end.
\end{proof}

\begin{lem}\label{lem:error}
Under the assumptions in Lemma \ref{lem:upbound1}, we have
\[
 \tau_i=O(u_{i}(x_i)^{-c_1+o(1)}),
\]
where $c_1=\min\{2, 2/(n-2\sigma)\}$.
Thus
\[
u_i(x_i)^{\tau_i}=1+o(1).
\]
\end{lem}

\begin{proof} Applying Theorem \ref{thm:harnack}, Theorem \ref{thm:schauder} and Proposition \ref{lem:upbound1} to $u_i$, we have
\be\label{eq:a3}
\|u_i\|_{C^2 (B_{3/2}(x_i)\setminus B_{1/2}(x_i))} \leq C R_i^{(n-2\sigma)\tau_i} m_i^{-\lda_i}.
\ee Choose $R_i\le m_i^{o(1)}$.  We write the equation \eqref{eq:main} of $u_i$ as
\be \label{eq:truncted}
u_i(x)= \int_{B_1(x_i)} \frac{K_i (y)u_i(y)^{p_i}}{|x-y|^{n-2\sigma}}\, \ud y +h_i(x),
\ee
where
\[
h_i(x):= \int_{|y-x_i|\ge 1} \frac{K_i (y)u_i(y)^{p_i}}{|x-y|^{n-2\sigma}}\, \ud y.
\]
Since $K_i$ and $u_i$ are nonnegative, by \eqref{eq:a3} and the computation in \eqref{eq:trick sub}, we have for any $x\in B_1(x_i)$
\be\label{eq:a4}
h_i(x) \leq 2^{n-2\sigma}\max_{\pa  B_1(x_i)}u_i \le C m_i^{-1+o(1)}
\ee
and
\be \label{eq:a5}
|\nabla h_i(x)|\le \begin{cases}
C\frac{|1-(1-|x-x_i|)^{2\sigma-1}|}{|2\sigma-1|} m_i^{-1+o(1)}\quad\mbox{if }\sigma\neq 1/2\\
C|\log(1-|x-x_i|)|m_i^{-1+o(1)}\quad\mbox{if }\sigma=1/2.
\end{cases}
\ee
Indeed, when $|x-x_i|<7/8$, it is easy to see that
\[
|\nabla h_i(x)|\le C \int_{|y-x_i|\ge 1} \frac{K_i (y)u_i(y)^{p_i}}{|x-y|^{n-2\sigma+1}}\, \ud y\le C\max_{\pa  B_1(x_i)}u_i \le  C m_i^{-1+o(1)}.
\]
When $7/8\le |x-x_i|\le 1$, then
\[
\begin{split}
|\nabla h_i(x)|&\le C \int_{|y-x_i|\ge 1, \ |y-x|\ge 1/8}+\int_{|y-x_i|\ge 1,\ |y-x|< 1/8}  \frac{K_i (y)u_i(y)^{p_i}}{|x-y|^{n-2\sigma+1}}\, \ud y\\
&\le C\max_{\pa  B_1(x_i)}u_i +Cm_i^{-p_i+o(1)}\int_{1-|x|<|y-x|<1/8}\frac{\ud y}{|x-y|^{n-2\sigma+1}}\\
&\le
\begin{cases}
C\frac{|1-(1-|x-x_i|)^{2\sigma-1}|}{|2\sigma-1|} m_i^{-1+o(1)}\quad\mbox{if }\sigma\neq 1/2\\
C|\log(1-|x-x_i|)|m_i^{-1+o(1)}\quad\mbox{if }\sigma=1/2.
\end{cases}
\end{split}
\]
Applying Proposition \ref{prop:pohozaev} to \eqref{eq:truncted} yields
\begin{align} \label{eq:a6}
& \tau_i \int_{B_1(x_i)} u_i(x)^{p_i+1} \,\ud x\nonumber\\& \leq C\Big(\int_{B_1(x_i)} |x-x_i| u_i(x)^{p_i+1} +(h_i(x)+|\nabla h_i(x)|)u_i(x)^{p_i}+ \int_{\pa B_1(x_i)} u_i^{p_i+1}\,  \ud s\Big).
\end{align}
By Proposition \ref{prop:blow up a bubble} and change of variables,
\begin{align}
\int_{B_1(x_i)} u_i(x)^{p_i+1} \,\ud x &\ge C^{-1}\int_{B_{r_i}(x_i)} \frac{m_i^{p_i+1}}{(1+|m_i^{(p_i-1)/2\sigma}(y-x_i)|^2)^{(n-2\sigma)(p_i+1)/2}}\,\ud y
\nonumber \\&
\ge C^{-1} m_i^{\tau_i(n/2\sigma-1)} \int_{R_i} \frac{1}{(1+|z|^2)^{(n-2\sigma)(p_i+1)/2}}\,\ud z\nonumber \\&
\ge C^{-1} m_i^{\tau_i(n/2\sigma-1)} ,
                            \label{eq:a7}
\end{align}
\[
 \int_{B_{r_i}(x_i)} u_i(x)^{p_i}\,\ud x \le C m_i^{-1+o(1)}
\]
and
\[
 \int_{B_{r_i}(x_i)} |x-x_i|u_i(x)^{p_i+1}\,\ud x\le  Cm_i^{-2/(n-2\sigma)+o(1)}.
\]
By Lemma \ref{lem:upbound1}, \eqref{eq:a4} and \eqref{eq:a5}
\[
 \int_{r_i<|x-x_i|<1} (h_i(x)+|\nabla h_i(x)|)u_i(x)^{p_i} \,\ud x\le Cm_i^{-2+o(1)},
\]
\[
 \int_{r_i<|x-x_i|<1} |x-x_i|u_i(x)^{p_i+1}\,\ud x\le  Cm_i^{-2/(n-2\sigma)+o(1)},
\]
and
\[
 \int_{\pa B_1(x_i)} u_i^{p_i+1}\,  \ud s \le Cm_i^{-2n/(n-2\sigma) +o(1)}.
\]
The lemma follows from the above estimates immediately. Consequently, if we choose $R_i\le m_i^{o(1)}$ as in the beginning of the proof, then we have that
\[
R_i^{(n-2\sigma)\tau_i}=1+o(1),
\]
which will be used later.
\end{proof}

Now we are ready to obtain our desired upper bound of $u_i$ near isolated simple blow up points.

\begin{prop}\label{prop:upbound2} Under the assumptions in Lemma \ref{lem:upbound1}, we have
\[
u_i(x)\leq Cu_i^{-1}(x_i)|x-x_i|^{2\sigma-n}\quad \mbox{for all } |x-x_i|\leq 1.
\]
\end{prop}

\begin{proof}  For $|x-x_i|\le r_i$, the proposition follows immediately from Proposition \ref{prop:blow up a bubble} and Lemma \ref{lem:error}.

We shall show first that
\be \label{eq:a7'}
u_i( \rho e+x_i)  u_i(x_i)\le C
\ee for any vector $|e|=1$. We assume the contrary that \eqref{eq:a7'} does not hold. Then along a subsequence we have
\[
\lim_{i\to\infty}u_i(\rho e+x_i) u_i(x_i)=+\infty.
\]

Since $u_i(x)\le A_3 |x-x_i|^{-2\sigma/(p_i-1)}$ in $B_2(x_i)$, it follows from  Lemma \ref{lem:harnack} that for any $0<\va<1$ there exists a positive constant 
$C(\va)$, depending on $n,\sigma, A_1, A_2, A_3$ and $\va$, such that 
\be \label{eq:extra1}
\sup_{B_{5/2}(x_i)\setminus B_\va(x_i)} u_i \le C(\va)\inf_{B_{5/2}(x_i)\setminus B_\va(x_i)} u_i. 
\ee

Let $\varphi_i (x)=u_i(\rho e+x_i) ^{-1} u_i(x)$. Then  for $|x-x_i|\le 1$,
\[
\varphi_i(x)=  \int_{\R^n} \frac{K_i(y) u_i(\rho e+x_i)^{p_i-1} \varphi_i(y)^{p_i}}{|x-y|^{n-2\sigma}}\,\ud y.
\]
Since $\varphi_i(\rho e+x_i)=1$,  by \eqref{eq:extra1} 
\be \label{eq:a8}
\|\varphi_i\|_{L^\infty(B_{3/2}(x_i)\setminus B_\va(x_i))} \le C(\va) \quad \mbox{for }0<\va<1.
\ee
Besides, 
by Lemma \ref{lem:upbound1},
\be\label{eq:a8'}
u_i(\rho e+x_i) ^{p_i-1} \to 0\quad\mbox{as } i\to \infty.  
\ee
Applying the local estimates in Section \ref{sec:linear integral equations} to $\varphi_i$, we conclude that  there exists $\varphi\in C^2(B_1\setminus \{0\})$ such that, after  passing to a subsequence, $\varphi_i\to \varphi$ in $C^2_{loc}(B_1\setminus \{0\})$. 

Let us write the equation of $\varphi_i$ as
\be
\varphi_i(x)=  \int_{B_1(x_i)} \frac{K_i(y) u_i(\rho e+x_i)^{p_i-1} \varphi_i(y)^{p_i}}{|x-y|^{n-2\sigma}}\,\ud y+h_i(x).
\ee
By  \eqref{eq:a8}, Theorem \ref{thm:schauder} and the proof of \eqref{eq:a5}, there exists $h\in C^2(B_1)$ such that, after passing to a subsequence, 
\be \label{eq:a9}
h_i(x) \to h(x)\ge 0\quad \mbox{in }C_{loc}^2(B_1).
\ee
Therefore,
\[
 \int_{B_1(x_i)} \frac{K_i(y) u_i(\rho e+x_i)^{p_i-1} \varphi_i(y)^{p_i}}{|x-y|^{n-2\sigma}}\,\ud y= \varphi_i(x)-h_i(x) \to \varphi(x) -h(x)\quad\mbox{in } C^2_{loc}(B_1\setminus \{0\}).
\]
 We shall evaluate what $G(x):=\varphi(x)-h(x)$ is.
For any $|x|>0$ and $2|x_i|\le \va<\frac{1}{2}|x-x_i|$,  in view of \eqref{eq:a8} and \eqref{eq:a8'}  we have
\begin{align*}
G(x)=& \lim_{i\to \infty} \int_{B_{\va}(x_i)} \frac{K_i(y) u_i(\rho e+x_i)^{p_i-1} \varphi_i(y)^{p_i}}{|x-y|^{n-2\sigma}}\,\ud y \\ &
 =|x|^{2\sigma-n}(1+O(\va)) \lim_{i\to \infty} \int_{B_{\va}(x_i)}K_i(y) u_i(\rho e+x_i)^{p_i-1} \varphi_i(y)^{p_i}\,\ud y\\&
 =|x|^{2\sigma-n}(1+O(\va)) a(\va),
\end{align*}
for some nonnegative function $a(\va)$ of $\va$.
Clearly, $a(\va)$ is nondecreasing, so $\lim_{\va\to 0} a(\va)$ exists which we denote as $a$.
Sending $\va\to 0$, we obtain
\[
G(x)= a |x|^{2\sigma-n}.
\]
Since $x_i\to 0$ is an isolated simple blow point, we have  $r^{\frac{n-2\sigma}{2}}\bar \varphi(r) \ge \rho^{\frac{n-2\sigma}{2}}\bar \varphi(\rho)$ for $0<r<\rho$. It follows that $\varphi$ is singular at $0$, and thus, $a>0$.  Hence,
\[
\lim_{i\to \infty} \int_{B_{1/8}(x_i)}K_i(y) u_i(\rho e+x_i)^{p_i-1} \varphi_i(y)^{p_i}\,\ud y>C\cdot G(e/2)>0.
\]
However,
\[
\begin{split}
&\int_{B_{1/8}(x_i)}K_i(y) u_i(\rho e+x_i)^{p_i-1} \varphi_i(y)^{p_i}\,\ud y\\
&\quad\le C u_i(\rho e+x_i)^{-1} \int_{B_{1/8}(x_i)}  u_i(y)^{p_i}\,\ud y\\
& \quad\le \frac{C}{u_i(\rho e+x_i) u_i(x_i)} \to 0 \quad\mbox{as } i \to \infty,
\end{split}
\]
 where we used Proposition \ref{prop:blow up a bubble},  Lemma \ref{lem:upbound1} and Lemma \ref{lem:error} in the second inequality.  This is a contradiction. 
 
Without loss of generality, we may assume that $\rho\le 1/2$.
It follows from Lemma \ref{lem:harnack} and \eqref{eq:a7'} that Proposition \ref{prop:upbound2} holds for $\rho\le |x-x_i| \le 1$.
To establish the inequality in the proposition for $r_i\le |x-x_i|\le \rho$, we only need to scale the problem and reduce it to the case of $|x-x_i|=1$.  Suppose the contrary that there exists a subsequence $\tilde x_i$ satisfying $r_i\le |\tilde x_i-x_i|\le \rho$
and $\lim_{i\to \infty}u_i(\tilde x_i) u_i(x_i)|\tilde x_i-x_i|^{n-2\sigma}= +\infty$.

Set $\tilde r_i:=|\tilde x_i-x_i|$, $\tilde u_i(x)= \tilde r_i^{2\sigma/(p_i-1)}u_i(x_i+\tilde r_i x)$. Then $\{u_i\}$ satisfies
\[
\tilde u_i(x)= \int_{\R^n} \frac{K_i(x_i+\tilde r_i y)\tilde u_i(y)^{p_i}}{|x-y|^{n-2\sigma}}\,\ud y \quad \mbox{for } x\in B_2,
\]
and all the hypotheses of Proposition \ref{prop:upbound2} with $0$ being an isolated simple blow up point. It follows from
\eqref{eq:a7'} that
\[
\tilde u_i(0) \tilde u_i(\frac{\tilde x_i-x_i}{\tilde r_i}) \le C.
\]
It follows (using Lemma \ref{lem:error}) that
\[
\lim_{i\to \infty} u_i(\tilde x_i) u_i(x_i) |\tilde x_i-x_i|^{n-2\sigma} <\infty.
\]
This is again a contradiction.

Therefore, the proposition is proved.
\end{proof}

\begin{cor}\label{cor:energy} Under the hypotheses of Lemma \ref{lem:upbound1}, we have
\[
\begin{split}
\int_{|x-x_i|\leq r_i}&|x-x_i|^{s}u_i(x)^{p_i+1}\\
&=\begin{cases}
O(u_i(x_i)^{-2s/(n-2\sigma)}),\quad& -n< s<n,\\
O(u_i(x_i)^{-2n/(n-2\sigma)}\log u_i(x_i)),\quad& s=n,\\
o(u_i(x_i)^{-2n/(n-2\sigma)}),\quad & s>n,
\end{cases}
\end{split}
\]
and
\[
\begin{split}
\int_{ r_i<|x-x_i|\leq1}&|x-x_i|^{s}u_i(x)^{p_i+1}\\&
=\begin{cases}
o(u_i(x_i)^{-2s/(n-2\sigma)}),\quad& -n< s<n,\\
O(u_i(x_i)^{-2n/(n-2\sigma)}\log u_i(x_i)),\quad& s=n,\\
O(u_i(x_i)^{-2n/(n-2\sigma)}),\quad & s>n.
\end{cases}
\end{split}
\]
\end{cor}
\begin{proof} Making use of Proposition \ref{prop:blow up a bubble}, Lemma \ref{lem:error} and Proposition
\ref{prop:upbound2}, the corollary follows immediately.
\end{proof}

By Proposition \ref{prop:upbound2} and its proof, we have the following corollary.

\begin{cor} \label{cor:convergence} Under the assumptions in Lemma \ref{lem:upbound1}, if we let $T_i(x)=T'_i(x)+T''_i(x)$, where
\[
T'_i(x):=u_i(x_i)\int_{B_1(x_i)} \frac{K_i(y) u_i(y)^{p_i}}{|x-y|^{n-2\sigma}}\,\ud y
\]
and
\[
T''_i(x):= u_i(x_i) \int_{\R^n\setminus B_1(x_i)} \frac{K_i(y) u_i(y)^{p_i}}{|x-y|^{n-2\sigma}}\,\ud y.
\] Then, after passing a subsequence,
\[
T'_i(x) \to a|x|^{2\sigma-n} \quad \mbox{in } C^2_{loc}(B_1\setminus \{0\})
\]
and
\[
T''_i(x) \to h(x) \quad \mbox{in } C_{loc}^2(B_1)
\] for some $h(x)\in C^2(B_2)$,
where
\be\label{eq:number a}
a=\Big(\frac{\pi^{n/2}\Gamma(\sigma)}{\Gamma(\frac n2 +\sigma)}\Big)^{-\frac{n}{2\sigma}}\int_{\R^n}\left(\frac{1}{1+|y|^2}\right)^\frac{n+2\sigma}{2}\ud y \lim_{i\to \infty} K_i(0)^{\frac{2\sigma-n}{2\sigma}}.
\ee
Consequently, we have
\[u_i(x_i) u_i(x)\to a|x|^{2\sigma-n} +h(x) \quad \mbox{in }C^2_{loc}(B_1\setminus \{0\}).
\]
\end{cor}

\begin{proof} Similar to the proof of Proposition \ref{prop:upbound2}, we set
$\varphi_i(x)=u_i(x_i) u_i(x)$, which satisfies
\begin{align*}
\varphi_i(x)&=\int_{\R^n} \frac{K_i(y) u_i(x_i)^{1-p_i} \varphi_i(y)^{p_i}}{|y-x|^{n-2\sigma}}\,\ud y\\&
= : \int_{B_1(x_i)} \frac{K_i(y) u_i(x_i)^{1-p_i} \varphi_i(y)^{p_i}}{|y-x|^{n-2\sigma}}\,\ud y +h_i(x)=T'_i(x) +T''_i(x) .
\end{align*}
We have all the ingredients as in the proof of Proposition \ref{prop:upbound2}. Hence, we only need to evaluate the positive constant $a$. For every $\va>0$, by Lebesgue dominated convergence theorem, we have
\begin{align*}
a\int_{B_{1/2}\setminus B_{1/4}} |x|^{2\sigma-n}\,\ud x& = \lim _{i\to \infty}  u_i(x_i)
\int_{B_{1/2}\setminus B_{1/4}} \left(\int_{B_{\va}(x_i)} \frac{K_i(y)u_i(y)^{p_i}}{|y-x|^{n-2\sigma}}\,\ud y\right)\ud x \\
& = \lim _{i\to \infty}  u_i(x_i)
\int_{B_{\va}(x_i)} K_i(y)u_i(y)^{p_i}\left(\int_{B_{1/2}\setminus B_{1/4}} \frac{1}{|y-x|^{n-2\sigma}}\,\ud x\right)\ud y.
 \end{align*}
After sending $\va\to 0$, we have
\[
a\int_{B_{1/2}\setminus B_{1/4}} |x|^{2\sigma-n}\,\ud x
=\int_{B_{1/2}\setminus B_{1/4}} |x|^{2\sigma-n}\,\ud x  \lim_{\va\to 0}\lim_{i\to \infty}   u_i(x_i) \int_{B_{\va}(x_i)} K_i (y)u_i(y)^{p_i}\,\ud y.
\]
Thus, by Proposition \ref{prop:blow up a bubble} and Lemma \ref{lem:error}, we have
\[
\begin{split}
a&=\lim_{i\to \infty}   u_i(x_i) \int_{B_{r_i}(x_i)} K_i (y)u_i(y)^{p_i}\,\ud y\\
&=\lim_{i\to \infty} K_i(0)k_i^{-\frac n2}\int_{\R^n}\left(\frac{1}{1+|y|^2}\right)^\frac{n+2\sigma}{2}\ud y\\
&=\Big(\frac{\pi^{n/2}\Gamma(\sigma)}{\Gamma(\frac n2 +\sigma)}\Big)^{-\frac{n}{2\sigma}}\int_{\R^n}\left(\frac{1}{1+|y|^2}\right)^\frac{n+2\sigma}{2}\ud y \lim_{i\to \infty} K_i(0)^{\frac{2\sigma-n}{2\sigma}},
\end{split}
\]
since
\[
\lim_{i\to \infty} \int_{B_\va\setminus B_{r_i}(x_i)}   u_i(x_i)  K_i (y)u_i(y)^{p_i}\,\ud y =0.
\]
\end{proof}

\subsection{Fine analysis of blow up with a flatness condition}

For later applications, we consider nonnegative solutions of
\be \label{eq:general}
u_i(x)= \int_{\R^n} \frac{K_i (y) H_i(y)^{\tau_i} (y) u_i(y)^{p_i}}{|x-y|^{n-2\sigma}}\,\ud y   \quad \mbox{for } x\in B_3
\ee
where $K_i$ satisfies \eqref{eq:Kca} with $\Omega=B_3$ as before, and $\{H_i\}\in C^{1,1} (B_3)$ satisfies
\be \label{eq:condH}
A_4^{-1} \le H_i(x) \le A_4 \quad \mbox{for } y\in B_3, \quad \|H_i\|_{C^{1,1}(B_3)} \le A_5
\ee
for some positive constants $A_4, A_5$. We assume that $x_i\to 0$ is an isolated simple blow up point of $\{u_i\}$ with constant $A_3$ and $\rho$, i.e.,
\[
|x-x_i|^{(p_i-1)/2\sigma} u_i(x) \le A_3,
\]
and $r^{\frac{p_i-1}{2\sigma}}\bar u_i(r)$ has precisely one critical point in $(0,\rho)$ for large $i$, where $\bar u_i(r)=\dashint_{\pa B_r(x_i)} u_i\,\ud s$.

\begin{lem} \label{lem:estimatetau1}  Suppose that $K_i$ satisfies the $(*)'_\beta$ condition in $B_3$ with $\beta<n$. Then we have
\[\begin{split}
\tau_i\leq &Cu_i(x_i)^{-2}+C|\nabla K_i(x_i)|u_i(x_i)^{-2/(n-2\sigma)}\\&
\quad +C(L(\beta,i)+L(\beta,i)^{\beta-1})u_i(x_i)^{-2\beta/(n-2\sigma)},
\end{split}
\]
where $C>0$ depends only on $n,\sigma,A_1,A_2,A_3,A_4,A_5,\beta$ and $\rho$.
\end{lem}
\begin{proof} Note that Proposition \ref{prop:blow up a bubble} and Proposition \ref{prop:upbound2} hold for the $u_i$ here. Let us write equation \eqref{eq:general} as
\be \label{eq:split}
u_i(x)= \int_{B_1(x_i)} \frac{K_i (y) H_i(y)^{\tau_i} (y) u_i(y)^{p_i}}{|x-y|^{n-2\sigma}}\,\ud y +h_i(x).
\ee
It follows from Proposition \ref{prop:upbound2} and \eqref{eq:a5} that for $x\in B_1(x_i)$
\be\label{eq:a6-1}
|\nabla h_i(x)|\le \begin{cases}
C\frac{|1-(1-|x-x_i|)^{2\sigma-1}|}{|2\sigma-1|} m_i^{-1}\quad\mbox{if }\sigma\neq 1/2\\
C|\log(1-|x-x_i|)|m_i^{-1}\quad\mbox{if }\sigma=1/2.
\end{cases}
\ee
Using Proposition \ref{prop:pohozaev} and Lemma \ref{lem:error}, by the same proof of Lemma \ref{lem:error} we have
\begin{align}
\tau_i &\le Cu_i(x_i)^{-2}+ C\left| \int_{B_1(x_i)}(x-x_i) \nabla (K_i H_i^{\tau_i})(x)u_i(x)^{p_i+1}\,\ud x\right|\nonumber \\ &
\le Cu_i(x_i)^{-2}+ C\tau_i \int_{B_1(x_i)} |x-x_i|u_i(x)^{p_i+1} \,\ud x \nonumber \\&
\quad + C\left| \int_{B_1(x_i)}(x-x_i) \nabla K_i(x) H_i^{\tau_i}u_i(x)^{p_i+1}\,\ud x\right|
                     \label{eq:a10}
\end{align}
By Corollary \ref{cor:energy}, the second term of the right  hand side is less than $\tau_i/2$ for $i$ large. It suffices to evaluate the third term, for which we will use the flatness condition.  We see
\[
\begin{split}
&\left|\int_{B_1(x_i)}( x-x_i)\nabla K_i H_i^{\tau_i} u_i^{p_i+1}\right|\\&
\leq C|\nabla K_i(x_i)|\int_{B_1(x_i)}|x-x_i|u_i^{p_i+1}\\&
\quad +C\int_{B_1(x_i)}|x-x_i||\nabla K_i(x)-\nabla K_i(x_i)|u_i^{p_i+1}\\&
\leq C|\nabla K_i(x_i)|u_i(x_i)^{-2/(n-2\sigma)}\\&
\quad  +C\int_{B_1(x_i)}|x-x_i||\nabla K_i(x)-\nabla K_i(x_i)|u_i^{p_i+1}.
\end{split}
\]
Recalling the definition of $(*)_\beta$, a direct computation yields
\begin{equation}\label{eq:lem4.5-(*)-1}
\begin{split}
&|\nabla K_i(x)-\nabla K_i(x_i)|\\&
\leq \Big\{\sum_{s=2}^{[\beta]}|\nabla ^s K_i(x_i)||x-x_i|^{s-1}+[\nabla ^{[\beta]}K_i]_{C^{\beta-[\beta]}(B_1(x_i))}|x-x_i|^{\beta-1}\Big\}\\&
\leq CL(\beta,i)\Big\{\sum_{s=2}^{[\beta]}|\nabla K_i(x_i)|^{(\beta-s)/(\beta-1)}|x-x_i|^{s-1}+|x-x_i|^{\beta-1}\Big\}.
\end{split}
\end{equation}
By Cauchy-Schwartz inequality, we have
\be\label{eq:lem4.5-(*)-2}
\begin{split}
&L(\beta,i)|\nabla K_i(x_i)|^{(\beta-s)/(\beta-1)}|x-x_i|^{s}\\&
\leq C(|\nabla K_i(x_i)||x-x_i|+(L(\beta,i)+L(\beta,i)^{\beta-1})|x-x_i|^{\beta}).
\end{split}
\ee
Hence, by Corollary \ref{cor:energy} we obtain
\be\label{eq:lem4.5-(*)-3}
\begin{split}
&\int_{B_1(x_i)}|x-x_i||\nabla K_i(x)-\nabla K_i(x_i)|u_i^{p_i+1}\\&
\leq   C|\nabla K_i(x_i)|u_i(x_i)^{-2/(n-2\sigma)}+C(L(\beta,i)+L(\beta,i)^{\beta-1})u_i(x_i)^{-2\beta/(n-2\sigma)}.
\end{split}
\ee
Lemma \ref{lem:estimatetau1} follows immediately.
\end{proof}

\begin{lem}\label{lem:gradient} Under the hypotheses of Lemma \ref{lem:estimatetau1},
\[
|\nabla K_i(x_i)|\leq Cu_i(x_i)^{-2}+C(L(\beta,i)+L(\beta,i)^{\beta-1}) u_i(x_i)^{-2(\beta-1)/(n-2\sigma)},
\]
where $C>0$ depends only on $n,\sigma,A_1,A_2,A_3,A_4,A_5,\beta$ and $\rho$.
\end{lem}

\begin{proof}
By \eqref{eq:split} and integrating by parts, we have
\begin{align*}
&\frac{1}{p_i+1}\int_{\pa B_1} x_j K_i(x) H_i(x)^{\tau_i} u_i^{p_i+1} \,\ud s-\frac{1}{p_i+1}\int_{B_1} \pa_j (K_i(x) H_i(x)^{\tau_i}) u_i^{p_i+1} \,\ud x\\
&=\frac{1}{p_i+1}\int_{B_1} K_i(x) H_i(x)^{\tau_i} \pa_j u_i^{p_i+1} \,\ud x\\&
= (2\sigma-n) \int_{B_1} K_i(x) H_i(x)^{\tau_i} u_i(x)^{p_i} \int_{B_1} \frac{(x_j-y_j)}{|x-y|^{n-2\sigma+2}} K_i(y) H_i(y)^{\tau_i} u_i(y)^{p_i} \,\ud y\ud x\\&
\quad +\int_{B_1} K_i(x) H_i(x)^{\tau_i} u_i(x)^{p_i} \pa_j h_i(x)\,\ud x \\&
=\int_{B_1} K_i(x) H_i(x)^{\tau_i} u_i(x)^{p_i} \pa_j h_i(x)\,\ud x.
\end{align*}
(If $\sigma<1/2$, then one can similar approximation arguments in the proof of Proposition \ref{prop:pohozaev}.) Hence,
\begin{align*}
& \left|\int_{B_1} \nabla (K_i(x) H_i(x)^{\tau_i}) u_i(x)^{p_i+1}\, \ud x\right| \\&
 \le C\big(  \int_{B_1} u_i^{p_i} |\nabla h_i|+ \int_{\pa B_1}  u_i^{p_i+1}\big).
\end{align*}
It follows from Proposition \ref{prop:upbound2} and Theorem \ref{thm:schauder} that  $u_i(x)\le C u_i(x_i)^{-1}$ for $1/2\leq |x| \le 1$,
and  \eqref{eq:a6-1} holds. Hence, by Corollary \ref{cor:energy}
\begin{align}
& \left|\int_{B_1} \nabla K_i(x) H_i(x)^{\tau_i} u_i(x)^{p_i+1}\, \ud x\right| \nonumber \\&
\le C( u_i(x)^{-2} +\tau_i\int_{B_1} u_i^{p_i+1}) \nonumber \\ &
\le C (u_i (x_i)^{-2} +\tau_i).
                               \label{eq:a11}
\end{align}
As that in proof of Lemma \ref{lem:estimatetau1}, we write
\[\begin{split}
&\int_{B_1} \nabla K_i(x) H_i(x)^{\tau_i} u_i(x)^{p_i+1}\, \ud x \\&
=\nabla K_i(x_i) \int_{B_1} H_i(x)^{\tau_i} u_i(x)^{p_i+1}\, \ud x +\int_{B_1} (\nabla K_i(x)-\nabla K_i(x_i)) H_i(x)^{\tau_i} u_i(x)^{p_i+1}\, \ud x,
\end{split}
\]
then use the triangle inequality and the $(*)_\beta$ condition (as in \eqref{eq:lem4.5-(*)-1}-\eqref{eq:lem4.5-(*)-2}) to obtain that
\begin{align*}
|\nabla K_i(x_i)| \le & Cu_i(x_i)^{-2}+C \tau_i +\frac{1}{4} |\nabla K_i(x_i)| \\&
+ C(L(\beta, i) +L(\beta, i)^{\beta-1})u_i(x_i)^{-2(\beta-1)/(n-2\sigma)}.
\end{align*}
Together with Lemma \ref{lem:estimatetau1}, the Lemma follows immediately.
\end{proof}

Combining the above two lemmas, we immediately obtain:

\begin{lem}\label{lem4.8} We have
\[
\tau_i\leq Cu_i(x_i)^{-2}+C(L(\beta,i)+L(\beta,i)^{\beta-1})u_i(x_i)^{-2\beta/(n-2\sigma)}.
\]
\end{lem}

\begin{cor}\label{cor4.1}
We further assume that
one of the following two conditions holds:
(i)
 \[
\beta=n-2\sigma \mbox{ and }L(\beta,i)=o(1),
\]
and
(ii)\[
\beta>n-2\sigma \mbox{ and }L(\beta,i)=O(1).
\]
Then for any $0<\delta<1$ we have
\[
\lim_{i\to\infty} u_i(x_i)^{2}\int_{B_{\delta}(x_i)}(x-x_i)\cdot \nabla(K_iH_i^{\tau_i})u_i^{p_i+1}=0.
\]
\end{cor}

\begin{proof}
By triangle inequalities, we have
\begin{align*}
&\left| \int_{B_{\delta}(x_i)}(x-x_i)\cdot \nabla(K_iH_i^{\tau_i})u_i^{p_i+1}\right|\\ &
\le \left|\int_{B_{\delta}(x_i)}(x-x_i)\cdot \nabla K_i H_i^{\tau_i} u_i^{p_i+1}\right|+\tau_i \left|\int_{B_{\delta}(x_i)}(x-x_i)\cdot \nabla H_i H^{\tau_i-1} K_i u_i^{p_i+1}\right|\\&
\le C|\nabla K_i(x_i)| \int_{B_\delta(x_i)} |x-x_i| u_i^{p_i+1} \\&
\quad +C \int_{B_\delta(x_i)} |x-x_i||\nabla K_i(x)-\nabla K_i(x_i)| u_i^{p_i+1} +C\tau_i \int_{B_\delta(x_i)} |x-x_i| u_i^{p_i+1}.
\end{align*}
Then the corollary follows from Lemma \ref{lem:gradient}, \eqref{eq:lem4.5-(*)-3}, Lemma \ref{lem4.8} and Corollary \ref{cor:energy}.
\end{proof}

\subsection{Isolated blow up points have to be isolated simple}

In this subsection, we will prove that under suitable assumptions on $K_i$, isolated blow up points have to be isolated simple.

\begin{prop}\label{prop5.1} Let $u_i\ge 0$ be a solution of \eqref{eq:general}, where $K_i$ satisfies \eqref{eq:Kca} and $(*)_{n-2\sigma}'$ condition with $\Omega=B_3$, and $H_i$ satisfies \eqref{eq:condH}.  Suppose that $x_i\to 0$ is an
isolated blow up point of $\{u_i\}$ in $B_2$ for some positive constant $A_3$, i.e.,
\[
|x-x_i|^{2\sigma/(p_i-1)} u_i(x) \le A_3\quad \mbox{in }B_2.
\] then $x_i\to 0$ is an
isolated simple blow up point.
\end{prop}

\begin{proof} Due to Proposition \ref{prop:blow up a bubble}, $r^{2\sigma/(p_i-1)}\overline u_i(r)$ has precisely
one critical point in the interval $0<r<r_i$,
where $r_i=R_iu_i(x_i)^{-\frac{p_i-1}{2\sigma}}$ as before.
Suppose $x_i\to 0$ is not an isolated simple blow up point and let $\mu_i$ be the second critical point of $r^{2\sigma/(p_i-1)}\overline u_i(r)$.
Then we see that
\be\label{5.2}
\mu_i\geq r_i,\quad \dlim_{i\to \infty}\mu_i=0.
\ee

Without loss of generality, we assume that $x_i=0$. Set
\[
\varphi_i(x)=\mu_i^{2\sigma/(p_i-1)}u_i(\mu_i x),\quad x\in \R^n.
\]
Clearly, $\varphi_i$ satisfies
\begin{align*}
\varphi_i(x)&=\int_{\R^n}\frac{\tilde K_i(y)\tilde H_i(y)^{\tau_i} \varphi_i(y)^{p_i}}{|x-y|^{n-2\sigma}}\,\ud y\quad\mbox{in } B_{3/\mu_i}
\\[2mm]
|x|^{2\sigma/(p_i-1)}\varphi_i(x)&\leq A_3\quad \mbox{for }|x|<2/\mu_i \to \infty,
\\[2mm]
\lim_{i\to \infty}\varphi_i(0)&=\infty,
\end{align*}
\[
r^{2\sigma/(p_i-1)}\overline \varphi_i(r)\mbox{ has precisely one critical point in } 0<r<1,
\]
and
\[
\frac{\mathrm{d}}{\mathrm{d}r}\left\{ r^{2\sigma/(p_i-1)}\overline \varphi_i(r)\right\}\Big|_{r=1}=0,
\]
where $\tilde K_i(y)=K_i(\mu_i y)$, $\tilde H_i(y)=H_i(\mu_i y)$ and $\overline \varphi_i(r)=|\pa B_r|^{-1}\int_{\pa B_r}\varphi_i$.
Therefore, $0$ is an isolated simple blow up point of $\varphi_i$.

We claim that
\be \label{eq:converg2}
\varphi_i(0) \varphi_i (x) \to \frac{a}{|x|^{n-2\sigma}} + a \quad \mbox{in }C^2_{loc}(\R^n \setminus \{0\}).
\ee
where $a$ is as in \eqref{eq:number a} with $K_i(0)$ replaced by $K_i(0)H_i(0)$.

By the equation of $\varphi_i$, we have for all $|x|\le 1/\mu_i$
\begin{align*}
\varphi_i(0)\varphi_i(x)&= \varphi_i(0)^{1-p_i}\int_{\R^n}
\frac{\tilde K_i(y) \tilde H_i(y)  (\varphi_i(0)\varphi_i(y))^{p_i}}{|x-y|^{n-2\sigma}}\,\ud y \\&
=\varphi_i(0)^{1-p_i}\left(\int_{B_t} + \int_{\R^n\setminus B_t}
\frac{\tilde K_i(y) \tilde H_i(y)  (\varphi_i(0)\varphi_i(y))^{p_i}}{|x-y|^{n-2\sigma}}\,\ud y\right),
\end{align*}
where $t>1$ is an arbitrarily fixed constant.
By the same proof of Corollary \ref{cor:convergence} we have, up to a subsequence,
\[
\varphi_i(0)^{1-p_i}\int_{B_t}
\frac{\tilde K_i(y) \tilde H_i(y)  (\varphi_i(0)\varphi_i(y))^{p_i}}{|x-y|^{n-2\sigma}}\,\ud y \to \frac{a}{|x|^{n-2\sigma}} \quad \mbox{in }C^2_{loc}(B_t \setminus \{0\}),
\]
and
\[
h_i(x):=  \varphi_i(0)^{1-p_i}\int_{\R^n\setminus B_t}
\frac{\tilde K_i(y) \tilde H_i(y)  (\varphi_i(0)\varphi_i(y))^{p_i}}{|x-y|^{n-2\sigma}}\,\ud y \to h(x)\ge 0\quad \mbox{in }C_{loc}^2(B_t).
\]
For any fixed large $R>0$, by Proposition \ref{prop:upbound2} and Lemma \ref{lem:harnack} we conclude that $\varphi_i(0)\varphi_i (x) \le C(t,R)$ for all $t\le |x| \le R$. Hence, for $x\in B_t$, we have
\[
\varphi_i(0)^{1-p_i}\int_{t \le |y|\le R}
\frac{\tilde K_i(y) \tilde H_i(y)  (\varphi_i(0)\varphi_i(y))^{p_i}}{|x-y|^{n-2\sigma}}\,\ud y \to 0\quad\mbox{as } i\to \infty. 
\]
Meanwhile, for any $x', x''\in B_t$, we have
\[
\begin{split}
\int_{\R^n\setminus B_R}
&\frac{\tilde K_i(y) \tilde H_i(y)  (\varphi_i(0)\varphi_i(y))^{p_i}}{|x'-y|^{n-2\sigma}}\,\ud y\\
&\le  \left( \frac{R+t}{R-t} \right)^{n-2\sigma}\int_{\R^n\setminus B_R}
\frac{\tilde K_i(y) \tilde H_i(y)  (\varphi_i(0)\varphi_i(y))^{p_i}}{|x''-y|^{n-2\sigma}}\,\ud y.
\end{split}
\]
Therefore, it follows that
\[
h(x') \le \left( \frac{R+t}{R-t} \right)^{n-2\sigma} h(x'').
\]
By sending $R\to \infty$ and exchanging the roles of $x'$ and $x''$, we have $h(x'') = h(x')$. Thus,
\[
h(x)\equiv h(0)\quad \mbox{for all } x\in B_t.
\]
 Since
\[
\frac{\mathrm{d}}{\mathrm{d}r}\left\{ r^{2\sigma/(p_i-1)}\varphi_i(0)\overline \varphi_i(r)\right\}\Big|_{r=1}=
\varphi_i(0)\frac{\mathrm{d}}{\mathrm{d}r}\left\{ r^{2\sigma/(p_i-1)}\overline \varphi_i(r)\right\} \Big|_{r=1}=0,
\]
we have, by  sending $i$ to $\infty$, that
\[h(0)=a>0.\]
Therefore, \eqref{eq:converg2} follows.

We are going to derive a contradiction to the Pohozaev identity in Proposition \ref{prop:pohozaev}. Since
$\frac{n-2\sigma}{2}-\frac{n}{p_i+1}\le 0$,  by Corollary \ref{cor4.1}, we have
\be \label{eq:b0}
\begin{split}
\lim_{i\to \infty} \varphi_i(0)^2 \Big(&(\frac{n-2\sigma}{2}-\frac{n}{p_i+1}) \int_{B_
\delta} \tilde K_i \tilde H_i^{\tau_i} \varphi_i^{p_i+1}\\& -\frac{1}{p_i+1} \int_{B_\delta} x \nabla(\tilde K_i \tilde H_i^{\tau_i}) \varphi_i ^{p_i+1} \Big) \le 0,
\end{split}
\ee
where $\tilde K_i $ satisfies the condition (i) of Corollary \ref{cor4.1} .

On the other hand, if we let
\[
b_i(x) :=\int_{\R^n\setminus B_\delta}  \frac{\tilde K_i(y) \tilde H_i(y)  \varphi_i(y)^{p_i}}{|x-y|^{n-2\sigma}}\,\ud y,
\]
then $\varphi_i(0) b_i(x)\ge h_i(x)\ge 2^{2\sigma-n}h_i(0)\to 2^{2\sigma-n} a$ for $x\in B_\delta$ provided that $\delta$ is small, and
\[
|\nabla b_i(x)|\le \begin{cases}
C\frac{|\delta^{2\sigma-1}-(\delta-|x-x_i|)^{2\sigma-1}|}{|2\sigma-1|} m_i^{-1}\quad\mbox{if }\sigma\neq 1/2\\
C|\log\delta-\log(\delta-|x-x_i|)|m_i^{-1}\quad\mbox{if }\sigma=1/2.
\end{cases}
\]
Hence, it follows from Proposition \ref{prop:blow up a bubble}, Lemma \ref{lem:error} and Proposition \ref{prop:upbound2} that
\begin{align}
 \int_{B_\delta} \tilde K_i \tilde H^{\tau_i}\varphi_i^{p_i} b_i& \ge C^{-1}a \varphi_i(0)^{-1} \int_{B_{\varphi_i(0)^{-(p_i-1)/2\sigma}}}\varphi_i^{p_i}\nonumber \\&
\ge C^{-1} \varphi_i(0)^{-2}\int_{B_1}(1+|x|^2)^{(2\sigma-n)/2},
 \label{eq:b1}
\end{align}
\be \label{eq:b2}
\left |\int_{B_\delta}x \nabla b_i \tilde K_i \tilde H^{\tau_i} \varphi_i^{p_i}\right| \le C\varphi_i(0)^{-1} \int_{B_\delta} |x| \varphi_i^{p_i}=o(1) \varphi_i(0)^{-2},
\ee
and
\be \label{eq:b3}
\int_{\pa B_\delta} \tilde K_i \tilde H_i^{\tau_i} \varphi_i^{p_i+1} \le C \varphi_i(0)^{-2n/(n-2\sigma)}.
\ee
Therefore,
\be\label{eq:b4}
\begin{split}
\lim_{i\to \infty} \varphi_i(0)^2 \Big(& \frac{n-2\sigma}{2} \int_{B_\delta} \tilde K_i \tilde H^{\tau_i}\varphi_i^{p_i} b_i \\& +
\int_{B_\delta}x \nabla b_i \tilde K_i \tilde H^{\tau_i} \varphi_i^{p_i}-\frac{1}{p_i+1} \int_{\pa B_\delta} \tilde K_i \tilde H_i^{\tau_i} \varphi_i^{p_i+1}  \Big)>0,
\end{split}\ee
which contradicts \eqref{eq:b0} and the Pohozaev identity in Proposition \ref{prop:pohozaev}.
\end{proof}

In Lemma \ref{lem:gradient}, we have obtained a decay estimate for  $|\nabla K_i(x_i)|$ as $i\to \infty$.
If one just wants to show $|\nabla K_i(x_i)|\to 0$, then it does not need $x_i\to 0$ to be simple nor $K_i$ satisfying the flatness condition.

 \begin{prop}\label{prop5.2}
Suppose the assumptions in Proposition \ref{prop5.1} hold except the $(*)'_{n-2\sigma}$ condition for $K_i$. Then
\[
|\nabla K_i(x_i)|\to 0\quad \mbox{as } i \to \infty.
\]
\end{prop}

\begin{proof}
Suppose that contrary that
\be\label{5.6}
|\nabla K_i(x_i)|\to d>0.
\ee
Without loss of generality, we assume $x_i=0$. There are two cases.

\medskip

\noindent \textit{Case 1.} $0$ is an isolated simple blow up point.
\medskip

In this case, we argue as in the proof of Lemma \ref{lem:gradient}.  Similar to \eqref{eq:a11}, we obtain
\[
 \left|\int_{B_1}\nabla K_iH_i^{\tau_i}u_i^{p_i+1}\right|\leq Cu_i^{-2}(0)+C\tau_i.
\]
It follows from the triangle inequality, uniform continuity of $\nabla K_i$ and Corollary \ref{cor:energy} that
\begin{align*}
 |\nabla K_i(0)|& \leq C\int_{B_1}|\nabla K_i(x)-\nabla K_i(0)|H_i^{\tau_i}u_i^{p_i+1}+o(1)=o(1).
\end{align*}

\medskip

\noindent\textit{Case 2.} $0$ is not an isolated simple blow up point.

\medskip

In this case we argue as in the proof of Proposition \ref{prop5.1}. The only difference is that we cannot derive \eqref{eq:b0} from Corollary \ref{cor4.1},
since $(*)'_{n-2\sigma}$ condition for $K_i$ is not assumed. Instead, we will use the condition \eqref{5.6} to show \eqref{eq:b0}.

Let $\mu_i, \varphi_i$, $\tilde{K_i}$ and $\tilde{H_i}$  be as in the proof of Proposition \ref{prop5.1}.
Now we estimate the integral term $\int_{B_\delta}\langle y,\nabla(\tilde K_i\tilde H_i^{\tau_i})\rangle \varphi_i^{p_i+1}$.
Arguing the same as in the proof of Lemma \ref{lem:estimatetau1} and using Corollary \ref{cor:energy}, we have
\[
\begin{split}
 \tau_i&\leq C\varphi_i(0)^{-2}+C\int_{B_\delta}|y||\nabla \tilde{K}_i(y)|H_i^{\tau_i} \varphi_i^{p_i+1}\\&
\leq C\varphi_i(0)^{-2}+C\mu_i\varphi_i(0)^{-2/(n-2\sigma)}.
\end{split}
\]
Similar to \eqref{eq:a11},
\[
 \left|\int_{B_\delta}\nabla \tilde K_i \tilde H_i^{\tau_i}\varphi_i^{p_i+1} \right|\leq C\varphi_{i}(0)^{-2}+C\tau_i.
\]
It follows that
\[
\begin{split}
 |\nabla \tilde K_i(0)|&\leq C\int_{B_\delta}|\nabla \tilde K_i(y)-\nabla \tilde K_i(0)|\varphi_i^{p_i+1}+C\varphi_i(0)^{-2}+C\tau_i\\
&\leq o(\mu_i)+C\varphi_i(0)^{-2}+C\tau_i.
\end{split}
\]
Since $|\nabla \tilde K_i(0)|=\mu_i |\nabla K_i(0)|\geq (d/2)\mu_i$, we have
\[
 \mu_i\leq C\varphi_i(0)^{-2}+C\tau_i.
\]
It follows that
\be\label{eq:tau-mu}
 \tau_i\leq C\varphi_i(0)^{-2}\quad \mbox{and}\quad \mu_i\leq C\varphi_i(0)^{-2}.
\ee
Therefore,
\[
 \left|\int_{B_\delta}\langle y,\nabla(\tilde K_i\tilde H_i^{\tau_i})\rangle \varphi_i^{p_i+1}\right|\leq C\varphi_i(0)^{-2-2/(n-2\sigma)},
\]
from which \eqref{eq:b0} follows.

Clearly, we have \eqref{eq:b1} - \eqref{eq:b3}. So we obtain a contradiction by the Pohozaev identity in Proposition \ref{prop:pohozaev}.

In conclusion, we complete the proof.
\end{proof}

\section{Global analysis on the sphere}

\label{section3}

Recall that the stereographic projection from $\Sn\backslash \{N\}$ to $\R^n$ is the inverse of
\[
F: \R^n\to \Sn\setminus\{N\}, \quad x\mapsto \left(\frac{2x}{1+|x|^2}, \frac{|x|^2-1}{|x|^2+1}\right),
\]
where $N$ is the north pole of $\Sn$, and its Jacobi determinant takes
\[
|J_F|=\Big(\frac{2}{1+|x|^2}\Big)^{n}.
\] Via the stereographic projection, the equation
\be \label{eq:onsphere}
v(\xi)=\int_{\Sn}\frac{K(\eta) v(\eta)^{p}}{|\xi-\eta|^{n-2\sigma}}\,\ud \eta \quad  \mbox{for }\xi \in \Sn
\ee
is translated to
\be \label{eq:inRn}
u(x)=\int_{\R^n} \frac{K(y)H(y)^\tau u(y)^{p}}{|x-y|^{n-2\sigma}}\,\ud y\quad \mbox{for }x\in \R^n,
\ee
where
\be  \label{eq:Hexp}
H(x):=|J_F(x)|^{\frac{n-2\sigma}{2n}}= \Big(\frac{2}{1+|x|^2}\Big)^{\frac{n-2\sigma}{2}},
\ee
 $u(x)=H(x)v(F(x))$, $K(x)=K(F(x))$ and $\tau=\frac{n+2\sigma}{n-2\sigma}-p\ge 0$.
% In the following, we will not always write the constant $c(n,\sigma)$.

\subsection{Finite blow up points and uniform energy bound}

We assume that $K$ satisfies
\be\label{eq:Kc1}
A_1^{-1}\leq K\leq A_1 \quad \mbox{on } \mathbb{S}^n,
\ee
and
\be\label{eq:Kc2}
\|K\|_{C^{1}(\Sn)}\leq A_2.
\ee
To ensure the solutions of \eqref{eq:onsphere} to be of $C^{2}$, we further assume that
\be \label{eq:Kc3}
\|K\|_{C^{1,1}(\Sn)} \le A_2, \quad \mbox{if }\sigma \le \frac{1}{2}.
\ee

\begin{prop}\label{prop:decomp}
Let $v\in C^2(\Sn)$ be a positive solution of \eqref{eq:onsphere}. For any $0<\va<1$ and $R>1$, there exist large positive constants $C_1^*$, $C_2^*$
depending on
$n,\sigma, A_1,A_2,\va$ and $R$ such that, if
\[
\max_{\mathbb{S}^n} v\geq C_1^*,
\]
then there exists an integer $1\le l=l(v)<\infty $ and a set
\[\Gamma_v=\{P_1,\dots, P_l\}\subset \Sn
\] $(P_j=P_j(v))$ such that\\
(i). $0\le \tau :=\frac{n+2\sigma}{n-2\sigma}-p<\va$, \\
(ii). If $P_1,\dots, P_l$ are local maximums of $v$ and for each $1\le j\le l$, let $\{z_1,\cdots,z_n\}$
be a geodesic normal coordinate system centered at
$P_j$, we have,
\be\label{eq:decomp1}
\|v^{-1}(P_j)v(v^{-\frac{(p-1)}{2\sigma}}(P_j)z)-(1+k|z|^2)^{(2\sigma-n)/2}\|_{C^2(B_{2R})}\leq \va,
\ee
where $k=\Big(\frac{K(P_j)\pi^{n/2}\Gamma(\sigma)}{\Gamma(\frac n2 +\sigma)}\Big)^{1/\sigma}$, and
\[
\{B_{Rv(P_j)^{-(p-1)/2\sigma}}(P_j)\} \quad \mbox{are disjoint balls}.
\]
(iii). $v(P)\leq C_2^*\{\mbox{dist}(P,\Gamma_v)\}^{-2\sigma/(p-1)}$ for all $P\in \mathbb{S}^n$.
\end{prop}

\begin{proof}
Given Theorem 1.4 in \cite{Li04} and the proof of Proposition \ref{prop:blow up a bubble}, the proof of Proposition \ref{prop:decomp} is similar to that  of Proposition 4.1 in \cite{Li95}  and Lemma 3.1 in \cite{SZ}, and is omitted here. We refer to \cite{Li95} and \cite{SZ} for details.
\end{proof}

\begin{prop}\label{prop:distlb} Assume the hypotheses in Proposition \ref{prop:decomp}.
Suppose that there exists some constant
$d>0$ such that $K$ satisfies $(*)'_{n-2\sigma}$ for some $L$ in $\om_d=\{P\in \mathbb{S}^n:|\nabla K(P)|<d\}$. Then, for
$\va>0$, $R>1$ and any solution $v$ of \eqref{eq:onsphere} with $\max_{\mathbb{S}^n}v>C_1^*$, we have
\[
 |P_1-P_2|\geq \delta^*>0,\quad \mbox{for any }P_1,P_2\in \Gamma_v\mbox{ and } P_1\neq P_2,
\]
where $\delta^*$ depends only on $n,\sigma,\va, R,A_1,A_2,L_2,d$, and the modulus of  continuity of $\nabla K$ if $\sigma>1/2$.
\end{prop}

\begin{proof} If we suppose the contrary, then there exist sequences of $\{p_i\}$ and $\{K_i\}$ satisfying the above assumptions,
and a sequence of corresponding solutions $\{v_i\}$ such that
\be\label{eq:e1}
\lim_{i\to \infty}|P_{1i}-P_{2i}|=0,
\ee
where $P_{1i},P_{2i}\in \Gamma_{v_i}$, and $|P_{1i}- P_{2i}|=\dmin_{\substack{P_1,P_2\in \Gamma_{v_i}
\\ P_1\neq P_2}}|P_1-P_2|$.

Since $B_{Rv_i(P_{1i})^{-(p_i-1)/2\sigma}}(P_{1i})$ and $B_{Rv_i(P_{2i})^{-(p_i-1)/2\sigma}}(P_{2i})$ have to be disjoint,
we have, because of \eqref{eq:e1}, that
$v_i(P_{1i})\to \infty$ and $v_i(P_{2i})\to \infty$. Therefore, we can pass to a subsequence (still denoted as $v_i$)
with $R_i\to \infty$, $\va_i\to 0$
as in Proposition \ref{prop:blow up a bubble} ($\va_i$ depending on $R_i$ can be chosen as small as we need in the following arguments) such that,
for $z$ being any geodesic normal coordinate system centered at $P_{ji},~j=1,2$, we have
\be\label{eq:e2}
 \|m_i^{-1}v_i(m_i^{-(p_i-1)/2\sigma}z)-(1+k_{ji}|z|^2)^{(2\sigma-n)/2}\|_{C^2(B_{2R_i}(0))}\leq \va_i,
\ee
where $m_i=v_i(0)$, $k_{ji}=\Big(\frac{K(P_{ji})\pi^{n/2}\Gamma(\sigma)}{\Gamma(\frac n2 +\sigma)}\Big)^{1/\sigma}$, $j=1,2; i=1,2,\dots$.

In the stereographic coordinates with $P_{1i}$ being the south pole, the equation \eqref{eq:onsphere} is transformed into
\be \label{eq:e3}
u_i(x)=\int_{\R^n}\frac{K_i(y)H(y)^{\tau_i} u_i(y)^{p_1}}{|x-y|^{n-2\sigma}}\,\ud y \quad \mbox{for all }x\in \R^n,
\ee
where $u_i(x)=H(x) v_i(F(x))$.  Let us still use $P_{2i}\in \R^n$ to denote the stereographic
coordinates of $P_{2i}\in \mathbb{S}^n$ and
set $\vartheta_i:=|P_{2i}|\to 0$. For simplicity, we
assume $P_{2i}$ is a local maximum point of $u_i$, since we can always
reselect a sequence of points to substitute $P_{2i}$.

From (ii) in Proposition \ref{prop:decomp}, there exists some constant $C$ depending only on $n,\sigma,$ such that
\be\label{eq:e4}
 \vartheta_i>\frac{1}{C}\max\{R_iu_i(0)^{-(p_i-1)/2\sigma}, R_iu_i(P_{2i})^{-(p_i-1)/2\sigma}\}.
\ee
Set
\[
 \varphi_i(y)=\vartheta_i^{2\sigma/(p_i-1)}u_i(\vartheta_i y)\quad \mbox{in } \R^n.
\]
It is easy to see that $\varphi_i$ is positive in $\R^n$ and satisfies
\be \label{eq:e5}
\varphi_i(x)=\int_{\R^n}\frac{\tilde K_i(y)\tilde H_i(y)^{\tau_i}\varphi_i(y)^{p_i}}{|x-y|^{n-2\sigma}}\,\ud y
\quad \mbox{for all } x\in\R^n,
\ee
where $\tilde K_i(x)=K_i(\vartheta_ix)$, $\tilde H_i(x)=H(\vartheta_i x)$. By Proposition \ref{prop:decomp},
we have $u_i$ satisfies
\[
\begin{split}
 u_i(x)&\leq C_2|x|^{-2\sigma/(p_i-1)}\quad \mbox{for all }|x|\leq \vartheta_i/2 \\
 u_i(y)&\leq C_2|y-P_{2i}|^{-2\sigma/(p_i-1)}\quad \mbox{for all }|y-P_{2i}|\leq \vartheta_i/2.
\end{split}
\]
In view of \eqref{eq:e4}, we therefore have
\[
\begin{split}
 \dlim_{i\to \infty}\varphi_i(0)=\infty,\quad &\dlim_{i\to \infty}\varphi_i(|P_{2i}|^{-1}P_{2i})=\infty,\\
|x|^{2\sigma/(p_i-1)}\varphi_i(x)\leq C^*_2,\quad &|x|\leq 1/2,\\
|x- |P_{2i}|^{-1}P_{2i}|^{2\sigma/(p_i-1)}\varphi_i(x)\leq C^*_2,\quad &|x-|P_{2i}|^{-1}P_{2i}|\leq 1/2.
\end{split}
\]
After passing a subsequence, if necessary,
there exists a point $\overline P\in \R^n$ with $|\overline P|=1$ such that
$|P_{2i}|^{-1}P_{2i}\to \overline P$ as $i\to \infty$. Hence $0$ and $\overline P$ are both isolated blow up points of
$\varphi_i$.

Actually, $0$ and $ |P_{2i}|^{-1}P_{2i} \to \overline P$  are both isolated simple blow up points of
$\varphi_i$. Indeed, If $|\nabla K_i(0)|< d/2$, then $0$ is an isolated simple blow up point of $\varphi_i$
because the $(*)_{n-2\sigma}$ condition holds in the region $\om_d$ and one can apply Proposition \ref{prop5.1} to conclude it.
If $|\nabla K_i(0)|\geq d/2$, arguing as in the proof of Proposition \ref{prop5.2} one can still conclude that $0$ is an isolated simple blow up point of $\varphi_i$. Similarly, $\overline P$ is also an isolated simple blow up point of $\varphi_i$.

It follows from Corollary \ref{cor:convergence} that for all $x\in \Big(B_{1/2}\cup B_{1/2}(\overline P)\Big)\setminus\{0,\overline P\}$
\be \label{eq:e6}
\varphi_i(0)\int_{B_{1/2}}\frac{\tilde K_i(y)\tilde H_i(y)^{\tau_i}\varphi_i(y)^{p_i}}{|x-y|^{n-2\sigma}}\,\ud y \to a(0)|x|^{2\sigma-n}
\ee
and
\be \label{eq:e7}
\varphi_i(|P_{2i}|^{-1}P_{2i})\int_{B_{1/2}(|P_{2i}|^{-1}P_{2i})}\frac{\tilde K_i(y)\tilde H_i(y)^{\tau_i}\varphi_i(y)^{p_i}}{|x-y|^{n-2\sigma}}\,\ud y \to a(\overline P)|x-\overline P|^{2\sigma-n},
\ee
where $a(0)>0$ and $a(\overline P)>0$ are as in \eqref{eq:number a}. It follows from Proposition \ref{prop:upbound2}, \eqref{eq:e6} and \eqref{eq:e7} that
\[
C^{-1}\le \varphi_i(0)\varphi_i(P_{2i}/(2|P_{2i}|))\le C,\quad C^{-1}\le \varphi_i(P_{2i}/|P_{2i}|)\varphi_i(P_{2i}/(2|P_{2i}|))\le C,\quad
\]
and thus,
\be \label{eq:e8}
C^{-1} \varphi_i(|P_{2i}|^{-1}P_{2i})\le  \varphi_i(0)\le  C \varphi_i(|P_{2i}|^{-1}P_{2i}),
\ee
where $C=C(n,\sigma, A_1, A_2)>1$. Therefore,
\begin{align} \label{eq:e9}
\varphi_i(0)\varphi_i(x)& =\varphi_i(0)\Big( \int_{B_{1/2}}  +\int_{B_{1/2}(|P_{2i}|^{-1}P_{2i})} \nonumber \\&
\quad \quad +\int_{\R^n\setminus (B_{1/2}\cup B_{1/2}(|P_{2i}|^{-1}P_{2i})} \frac{\tilde
K_i(y)\tilde H_i(y)^{\tau_i}\varphi_i(y)^{p_i}}{|x-y|^{n-2\sigma}}\,\ud y\Big ) \nonumber \\&
 \to a(0) |x|^{2\sigma-n}+ a(\overline P)' |x-\overline P|^{2\sigma-n } + h(x),
\end{align}
where $a(\overline P)' \ge C^{-1} a(\overline P)>0$ by \eqref{eq:e8}, and $0\le h(x)\in C^2(B_{1/2}\cup B_{1/2}(\overline P))$.

If we let
\[
b_i(x):=\int_{\R^n\setminus B_{1/2}}   \frac{\tilde
K_i(y)\tilde H_i(y)^{\tau_i}\varphi_i(y)^{p_i}}{|x-y|^{n-2\sigma}}\,\ud y,
\]
then \[\lim_{i\to\infty}\varphi_i(0) b_i(x)=a(\overline P)' |x-\overline P|^{2\sigma-n } + h(x) \ge c_0>0\quad \mbox{in } B_{1/2}.\]

Thus, if $|\nabla K_i(0)|<d/2$, we can obtain
a contradiction by using the Phozaev identity as \eqref{eq:b0}-\eqref{eq:b4} in the proof of Proposition \ref{prop5.1}. In the case that  $|\nabla K_i(0)|\ge d/2$, we still can derive a contradiction since $\vartheta_i\le C\varphi_i(0)^{-2}$, $\tau_i\le C\varphi_i(0)^{-2}$ as in the proof of \eqref{eq:tau-mu} in Proposition \ref{prop5.2}.

Therefore, the proposition is proved.
\end{proof}

\subsection{Compactness}

Consider
\be\label{6.11}
\begin{split}
P_\sigma(v_i)=c(n,\sigma)K_i v_i^{p_i}\quad &\mbox{on } \mathbb{S}^n,\\
v_i>0,\quad &\mbox{on } \mathbb{S}^n,\\
p_i=\frac{n+2\sigma}{n-2\sigma}-\tau_i, \quad &\tau_i\geq 0,\tau_i\to 0.
\end{split}
\ee
By \eqref{P sigma inverse}, \eqref{6.11} is equivalent to
\be \label{eq:onsphere rescaled}
v_i(\xi)=\frac{\Gamma(\frac{n+2\sigma}{2})}{2^{2\sigma}\pi^{n/2}\Gamma (\sigma)}\int_{\Sn}\frac{K_i(\eta) v_i(\eta)^{p}}{|\xi-\eta|^{n-2\sigma}}\,\ud \eta \quad  \mbox{for }\xi \in \Sn.
\ee

\begin{thm}\label{thm:energy bound}
 Suppose $K_i$ satisfies the assumptions of $K$ in Proposition \ref{prop:distlb}. Let $v_i$ be solutions of
\eqref{6.11}. Then we have
\be\label{6.12}
\|v_i\|_{H^\sigma(\Sn)}\leq C,
\ee
where $C>0$ depends only on $n,\sigma,A_1,A_2,L,d$ and  the modulus of  continuity of $\nabla_{g_{\Sn}} K_i$ if $\sigma>1/2$. Furthermore,
after passing to a subsequence, either $\{v_i\}$ stays bounded in $L^\infty(\mathbb{S}^n)$
or $\{v_i\}$ has only isolated simple blow up points and the distance between any two blow up points is bounded blow by some
positive constant depending only on $n,\sigma,A_1,A_2,L,d$ and  the modulus of  continuity of $\nabla K_i$ if $\sigma>1/2$.
\end{thm}

\begin{proof}
The theorem follows immediately from Proposition \ref{prop:distlb}, Proposition \ref{prop5.2}, Proposition \ref{prop:blow up a bubble}, Proposition \ref{prop5.1}
and Corollary \ref{cor:energy}.
%Proposition \ref{5.2} guarantee that $K$ satisfies that all the blow up points are near the critical point of K and hence we only need to assume (*)_{\beta} condition for K near K's critical points, like in the assumption of Proposition \ref{6.2}. Prop \ref{prop5.1} tells us that isolated blow up point are simples. and hence Prop \ref{prop4.1} and lemma \ref{lem4.5} shows the energy near blow up points are bounded. then Proposition \ref{prop6.2} tells us that the blow up points have a universal distance between each other, and hence the total energy are bounded.
\end{proof}

\begin{proof}[Proof of Theorem \ref{main thm B}] It follows immediately from Theorem \ref{thm:energy bound}.
\end{proof}

In the next theorem, we impose a stronger condition on $K_i$ such that $\{u_i\}$ has at most one blow up point.

\begin{thm}\label{thm:blow up one pt}
Suppose the assumptions in Theorem \ref{thm:energy bound}.
Suppose further that $\{K_i\}$ satisfies $(*)'_{n-2\sigma}$
condition for some sequences $L(n-2\sigma,i)=o(1)$
in $\om_{d,i}=\{q\in \mathbb{S}^n: |\nabla_{g_0}K_i|<d\}$ or $\{K_i\}$ satisfies $(*)'_{\beta}$
condition with $\beta\in (n-2\sigma,n)$ in
$\om_{d,i}$. Then, after passing to a subsequence, either $\{v_i\}$ stays bounded in $L^\infty(\mathbb{S}^n)$
or $\{v_i\}$ has precisely one isolated simple blow up point.
\end{thm}

\begin{proof}
The strategy is the same as the proof of Proposition \ref{prop:distlb}. We assume there
are two isolated blow up points. After some transformation,
we can assume that they are in the same half sphere.
The condition of $\{K_i\}$ guarantees that Corollary \ref{cor4.1} holds for $u_i$, where $u_i$ is as below \eqref{eq:Hexp}. Hence \eqref{eq:b0} holds for $u_i$. Meanwhile, \eqref{eq:b4} for $u_i$ is also valid, since the distance between these blow up points is uniformly lower bounded which is due to Proposition \ref{prop:distlb}. This reaches a contradiction.
\end{proof}

\begin{proof}[Proof of Theorem \ref{main thm C}]
It follows from Theorem \ref{thm:blow up one pt}.
\end{proof}

\begin{thm}\label{thm: 1-2 pts} Let $v_i$ be positive solutions of \eqref{6.11}.
Suppose that $\{K_i\}\subset C^{1}(\Sn)$ satisfies \eqref{eq:Kc2}-\eqref{eq:Kc3},
and for some point $P_0\in \Sn$, $\va_0>0$, $A_1>0$ independent of $i$ and $1<\beta<n$, that
\[
\{K_i\} \text{ is bounded in } C^{[\beta],\beta-[\beta]}(B_{\va_0}(q_0)),\quad \quad K_i(P_0)\geq A_1
\]
and
\[
K_i(y)=K_i(0)+Q_i^{(\beta)}(y)+R_i(y), \quad |y|\leq \va_0,
\]
where $y$ is a geodesic normal coordinates system centered at $P_0$,
$Q_i^{(\beta)}(y)$ satisfies $Q_i^{(\beta)}(\lda y)=\lda^{\beta} Q_i^{(\beta)}(y)$, $\forall\lda>0,\ y\in\R^n$,  and $R_i(y)$ satisfies
\[\sum_{s=0}^{[\beta]} |\nabla^s R_i(y)||y|^{-\beta+s}\to 0\] uniformly in $i$ as $y\to 0$.

Suppose also that $Q_i^{(\beta)}\to Q^{(\beta)}$ in $C^1(\mathbb{S}^{n-1})$ and for some positive constant $A_6$ that
\be\label{6.13}
A_6|y|^{\beta-1}\leq |\nabla Q^{(\beta)}(y)|,\quad |y|\leq \va_0,
\ee
and
\be\label{6.14}
\left(
\begin{array}{l}
 \int_{\R^n}\nabla Q^{(\beta)}(y+y_0)(1+|y|^2)^{-n}\,\ud y\\[2mm]
\int_{\R^n}Q^{(\beta)}(y+y_0)(1+|y|^2)^{-n}\,\ud y
\end{array} \right)\neq 0, \quad \forall\  y_0\in \R^n.
\ee
If $P_0$ is an isolated simple blow up point of $v_i$, then
$v_i$ has to have at least another blow up point.
\end{thm}

\begin{proof}
The proof is by checking the ``balance" condition of Kazdan-Warner type \eqref{eq:KZ}. It is the same as that of Theorem 5.3 in \cite{JLX}.
\end{proof}

\begin{thm}\label{thm:maximum bound}
Suppose that $K\in C^{1}(\Sn)$ ($K\in C^{1,1}(\Sn)$ if $0<\sigma\le 1/2$), for some constant $A_1>0$,
\[
 1/A_1\leq K(\xi)\leq A_1\quad  \mbox{for all }\xi\in \Sn.
\]
Suppose also that for any critical point $\xi_0$ of $K$ and 
a geodesic normal coordinates system $\{y_1, \cdots, y_n\}$ centered at $\xi_0$, there exist
some small neighborhood $\mathscr{O}$ of $0$, a positive constant $L$, and $\beta=\beta(\xi_0)\in(n-2\sigma, n)$ such that
\[
 \|\nabla^{[\beta]}K\|_{C^{\beta-[\beta]}(\mathscr{O})}\leq L
\]
and
\[
 K(y)=K(0)+Q_{(\xi_0)}^{(\beta)}(y)+R_{(\xi_0)}(y)\quad \mbox{in }\mathscr{O},
\]
where $Q_{\xi_0}^{(\beta)}(y)\in C^{[\beta]-1,1}(\mathbb{S}^{n-1})$ satisfies
$Q_{\xi_0}^{(\beta)}(\lda y)=\lda^{\beta}Q_{\xi_0}^{(\beta)}(y)$, $\forall \lda >0$, $y\in \R^n$, and for some positive constant $A_6$
\[
A_6|y|^{\beta-1}\leq |\nabla Q_{\xi_0}^{(\beta)}(y)|,\quad y\in \mathscr{O},
\]
and
\[
\left(
\begin{array}{l}
 \int_{\R^n}\nabla Q_{\xi_0}^{(\beta)}(y+y_0)(1+|y|^2)^{-n}\,\ud y\\[2mm]
\int_{\R^n}Q_{\xi_0}^{(\beta)}(y+y_0)(1+|y|^2)^{-n}\,\ud y
\end{array} \right)\neq 0, \quad \forall\  y_0\in \R^n,
\]
and $R_{\xi_0}(y)\in C^{[\beta]-1,1}(\mathscr{O})$ satisfies $\lim_{y\to 0}\sum_{s=0}^{[\beta]}|\nabla^sR_{\xi_0}(y)||y|^{-\beta+s}=0$.

Then there exists a positive constant $C\geq 1$ depending on $n,\sigma, K$ such that for every solution $v$ of \eqref{main equ} there holds
\[
 1/C\leq v\leq C,\quad \mbox{on }\Sn.
\]
\end{thm}

\begin{proof}
It follows directly from Theorem \ref{thm:blow up one pt} and Theorem \ref{thm: 1-2 pts}.
\end{proof}

\begin{proof}[Proof of the compactness part of Theorem \ref{main thm A}]
It is easy to check that, if $K$ satisfies the condition in Theorem \ref{main thm A}, then it must satisfy the condition in the above theorem.
Therefore, we have the lower and upper bounds of $v$. The $C^2$ norm bound of $v$ follows immediately.
\end{proof}

\section{Existence}
\label{section4}

\subsection{Positivity of minimizers}

Let $f\in L^p(\Sn)$ for $p\in (1,\infty)$.  We say $v\in L^1(\Sn)$ is a weak solution of
\be\label{eq:lineareq}
P_\sigma v=f\quad \mbox{on }\Sn,
\ee
if for any $\varphi\in C^\infty(\Sn)$ there holds
\be\label{eq:weaksol}
\int_{\Sn} v P_\sigma(\varphi)=\int_{\Sn} f \varphi.
\ee
On the other hand, for any $g\in C^\infty(\Sn)$ if we let
\[
\varphi(\xi)= I^{\sigma}(g)= c_{n,\sigma}\int_{\Sn} \frac{g}{|\xi-\eta|^{n-2\sigma}}\,\ud \eta,
\]
then by \eqref{P sigma inverse} and  Fubini theorem, we have
\[
\int_{\Sn} vg= \int_{\Sn} I^{\sigma} (f) g,
\]
This implies
\be\label{eq:lineareq2}
v(\xi)=I^{\sigma}(f)(\xi)= c_{n,\sigma}\int_{\Sn} \frac{f}{|\xi-\eta|^{n-2\sigma}}\,\ud \eta\quad a.e. ~\xi \in \Sn.
\ee
Consequently, if $f\ge 0$ on $\Sn$ then $v\ge 0$ on $\Sn$, and if in addition that $f\not\equiv 0$ then $v>0$ on $\Sn$. By Riesz potential theory, $v\in H^{\sigma, p}(\Sn)$ if $p>1$.

\begin{prop} \label{prop:subcrit} Let $K$ satisfy \eqref{eq:Kc1}-\eqref{eq:Kc3}, $1<p<\frac{n+2\sigma}{n-2\sigma} $ and
\[
Q_p[v]:=\frac{\int_{\Sn}v P_\sigma v}{(\int_{\Sn}K |v|^{p+1})^{2/(p+1)}}\quad \mbox{for } v\in H^\sigma(\Sn) \setminus \{0\}.
\] Then the minimum
\[
Q_p:=\inf_{v\in H^\sigma(\Sn)} Q_p[v]
\]
is achieved by some positive function $v_p$ in $H^\sigma(\Sn)$.
\end{prop}

\begin{proof} Since the Sobolev embedding $H^\sigma \hookrightarrow L^{p+1}$ is compact, the existence of minimizer $v_p$ follows from standard variational arguments. Also, it satisfies that
\[
P_{\sigma}(v_p)=Q_pK |v_p|^{p-1}v_p\quad\mbox{and}\quad \int_{\Sn}K|v_p|^{p+1}=1.
\]
We only need to show the positivity of $v_p$, which follows from the proof of Proposition 5 in \cite{Ro10}.

Indeed, let $w_p$ be such that $P_\sigma(w_p)=|P_\sigma(v_p)|=Q_pK|v_p|^p$. This implies that $P_{\sigma}(w_p\pm v_p)\ge 0$, and thus, $w_p\ge \pm v_p$, i.e., $w_p\ge |v_p|$. Since $v_p\not\equiv 0$, we have $w_p>0$. Since $Q_p$ is the minimum,
\[
\begin{split}
Q_p\le \frac{\int_{\Sn}w_p P_\sigma (w_p)}{(\int_{\Sn}K |w_p|^{p+1})^{2/(p+1)}}&=\frac{\int_{\Sn}w_p Q_pK|v_p|^p}{(\int_{\Sn}K |w_p|^{p+1})^{2/(p+1)}}\\
&\le Q_p\frac{\left(\int_{\Sn} K|v_p|^{p+1}\right)^{\frac{p}{p+1}}\left(\int_{\Sn} K|w_p|^{p+1}\right)^{\frac{1}{p+1}} }{(\int_{\Sn}K |w_p|^{p+1})^{2/(p+1)}}\\
&\le Q_p.
\end{split}
\]
It follows that all these inequalities are actually equalities. Thus, $|v_p|=w_p>0$, which means that $v_p$ does not change sign. So we may assume that $v_p$ is positive.
\end{proof}

\begin{prop} Let $v\in H^\sigma(\Sn)$ be a solutions of
\[
P_\sigma v= K |v|^{\frac{4\sigma}{n-2\sigma}} v \quad \mbox{on }\Sn,
\]
where $0\le K\le A_1$ on $\Sn$.
There exists $\va_0=\va_0(n,\sigma, A_1)>0$ such that if the negative part of $v$ satisfies $\int_{\Sn} |v^{-}|^{\frac{2n}{n-2\sigma}}\le \va_0$, then $v\ge0$.
\end{prop}

\begin{proof} It follows that $v$ satisfies
\[
v(\xi)= c_{n,\sigma}\int_{\Sn} \frac{K(\eta) |v|^{\frac{4\sigma}{n-2\sigma} }v}{|\xi-\eta|^{n-2\sigma}}\,\ud \eta.
\]
Multiplying both sides by $-(v^-)^{\frac{n+2\sigma}{n-2\sigma}}$ and integrating over $\Sn$, we see that
\[
\begin{split}
\int_{\Sn} (v^{-})^{\frac{2n}{n-2\sigma}} & =-c_{n,\sigma}\int_{\Sn} \int_{\Sn} \frac{K(\eta) |v(\eta)|^{\frac{4\sigma}{n-2\sigma}} v(\eta) (v^-(\xi))^{\frac{n+2\sigma}{n-2\sigma}} }{|\xi-\eta|^{n-2\sigma}}\,\ud \eta\ud \xi \\&
\le
c_{n,\sigma}\int_{\Sn} \int_{\Sn} \frac{K(\eta) (v^-(\eta))^{\frac{n+2\sigma}{n-2\sigma}}  (v^-(\xi))^{\frac{n+2\sigma}{n-2\sigma}} }{|\xi-\eta|^{n-2\sigma}}\,\ud \eta\ud \xi\\&
\le C(n,\sigma,A_1) \|(v^-)^{\frac{n+2\sigma}{n-2\sigma}}\|_{L^\frac{2n}{n+2\sigma}}^2  \\&
= C(n,\sigma,A_1) \left(\int_{\Sn} (v^{-})^{\frac{2n}{n-2\sigma}} \right)^{\frac{n+2\sigma}{n}},
\end{split}
\]
where in the second inequality we used Hardy-Littlewood-Sobolev inequality on $\Sn$ (see, e.g., \cite{Lie83}). Therefore, if $\va_0< C(n,\sigma,A_1)^{-n/2\sigma} $, then $v^- \equiv 0$.
\end{proof}

\subsection{Existence for antipodally symmetric functions}

In this section,  we shall prove Theorem \ref{K-M-E-S}, which is for antipodally symmetric functions $K$.

Let $H^{\sigma}_{as}$ be the set of antipodally symmetric functions in $H^{\sigma}(\mathbb{S}^n)$, and let
\[
\lambda_{as}(K)=\inf_{v\in H^{\sigma}_{as}}\left\{\int_{\mathbb{S}^n}vP_{\sigma}(v): \int_{\mathbb{S}^n}K|v|^{\frac{2n}{n-2\sigma}}=1\right\}.
\]
We also denote $\omega_n$ as the volume of $\mathbb{S}^n$.
The proof of Theorem \ref{K-M-E-S} is divided into two steps.

\begin{prop}\label{less then exist1}
Let $0\le K\in C^0(\Sn)$ be antipodally symmetric. If
\begin{equation}\label{eq:less1}
\lambda_{as}(K)<\frac{P_{\sigma}(1)\omega_n^{\frac{2\sigma}{n}}2^{\frac{2\sigma}{n}}}{(\max_{\Sn}{K})^{\frac{n-2\sigma}{n}}},
\end{equation}
then there exists a positive and antipodally symmetric $C^{2\sigma^*}(\mathbb{S}^n)$ solution of \eqref{main equ}, where $\sigma^*=\sigma$ if $2\sigma \notin \mathbb{N}^+$ and otherwise $0<\sigma^*<\sigma$.
\end{prop}

\begin{prop}\label{less by test function1}
Let $0\le K\in C^0(\Sn)$ be antipodally symmetric. If there exists a maximum point of $K$ at which $K$ has flatness order greater than $n-2\sigma$, then
\begin{equation}\label{less than1}
\lambda_{as}(K)<\frac{P_{\sigma}(1)\omega_n^{\frac{2\sigma}{n}}2^{\frac{2\sigma}{n}}}{(\max_{\Sn}{K})^{\frac{n-2\sigma}{n}}}.
\end{equation}
\end{prop}

\begin{proof}[Proof of Theorem \ref{K-M-E-S}]
It follows from Proposition \ref{less then exist1} and Proposition \ref{less by test function1}.
\end{proof}

The proof of Proposition \ref{less then exist1} uses subcritical approximations. For $1<p<\frac{n+2\sigma}{n-2\sigma}$, we define
\[
\lambda_{as,p}(K)=\inf_{v\in H^{\sigma}_{as}}\left\{\int_{\mathbb{S}^n}vP_{\sigma}(v): \int_{\mathbb{S}^n}K|v|^{p+1}=1\right\}.
\]
We begin with a lemma
\begin{lem}\label{lem of existence subcritical}
Assume $K$ as that in Proposition \ref{less then exist1}.  Then $\lambda_{as,p}(K)$ is achieved by
a positive and antipodally symmetric $C^{2\sigma^*}(\Sn)$ function $v_{p}$, which satisfies
\begin{equation}\label{eq:subcritical1}
P_{\sigma}(v_p)=\lambda_{as,p}(K) K v_p^p\ \ \ \text{and}\ \ \ \int_{\mathbb{S}^n}Kv_p^{p+1}=1.
\end{equation}
\end{lem}

\begin{proof}
The existence of a positive solution $v_p$ in $H^\sigma(\Sn)$ can be found by considering the minimizing problem as that in Proposition \ref{prop:subcrit}.
The regularity of $v_p$ follows from Theorem \ref{thm:li}, Theorem \ref{thm:harnack} and a standard result of Riesz potential.
\end{proof}

\begin{proof}[Proof of Proposition \ref{less then exist1}]The proof is similar to that of Proposition 2.1 in \cite{JLX2}, and here we just sketch the proof.
First of all, we have
\[
\limsup_{p\to\frac{n+2\sigma}{n-2\sigma}}\lambda_{as,p}(K)\leq \lda_{as}(K).
\]
Hence, we may assume that there exists a sequence $\{p_i\}\to\frac{n+2\sigma}{n-2\sigma}$ such
that $\lambda_{as,p_i}(K)\to \lambda$ for some $\lambda\leq\lambda_{as}(K)$. Since $\{v_{i}\}$, which is a sequence of minimizers in Lemma \ref{lem of existence subcritical} for $p=p_i$, is bounded in
$H^{\sigma}(\mathbb{S}^n)$, then there exists $v\in H^{\sigma}(\Sn)$ such that $v_{i}\rightharpoonup v$ weakly in  $H^{\sigma}(\mathbb{S}^n)$ and $v$ is nonnegative. If $v\not\equiv 0$, it follows from the integral equation of $v$ that $v>0$ on $\Sn$, and we are done. Now we suppose that $v\equiv 0$. If $\{\|v_{i}\|_{L^{\infty}(\mathbb{S}^n)}\}$ is bounded, by Theorem \ref{thm:harnack} we have $\{\|v_{i}\|_{C^{2\sigma^*}(\mathbb{S}^n)}\}$ is bounded, too. Therefore, $v_{i}\to 0$ in $C^0(\Sn)$ which leads to $1=\int_{\mathbb{S}^n}K|v_{i}|^{p_i+1}\to 0$. This is a contradiction. Thus we may assume that $v_{i}(x_{i}):=\max_{\mathbb{S}^n}v_{i}\to\infty.$ Since $\Sn$ is compact, there exists a
subsequence of $\{x_{i}\}$, which will be still denoted as $\{x_{i}\}$, and $\bar x$ such that $x_{i}\to\bar x$.
Without loss of generality we assume that $\bar x$ is the south pole. Via the stereographic projection $F^{-1}$, \eqref{eq:subcritical1} becomes
\begin{equation}\label{eq:subcritical in plane1}
u_{i}(y)=c_{n,\sigma}\lambda_{as,p_i}(K) \int_{\R^n}\frac{\hat K_i(z) u_{i}^{p_i}(z)}{|y-z|^{n-2\sigma}}\,\ud z \quad y\in\mathbb{R}^n,
\end{equation}
where $v_{i}\circ F(y)=(\frac{1+|y|^2}{2})^{\frac{n-2\sigma}{2}}u_{i}(y)$, $\hat K_i(y)= H^{\frac{n+2\sigma}{n-2\sigma}-p_i}(y) K(F(y))$, and $H$ is as in \eqref{eq:Hexp}.

Thus for any $y\in\mathbb{R}^n$, $u_{i}(y)\leq 2^{\frac{n-2\sigma}{2}}u_{i}(y_{i})$
where $y_{i}:=F^{-1}(x_{i})\to 0.$ For simplicity, we denote $m_{i}:=u_{i}(y_{i})$. By our assumption on $v_{i}$ we have $m_{i}\to\infty.$
Define
\[
\tilde u_{i}(y)=m_{i}^{-1}u_{i}\big(m_{i}^{\frac{1-p_i}{2\sigma}}y+y_{i}\big).
\]
From \eqref{eq:subcritical in plane1} we see that $\tilde u_{i}(y)$ satisfies
\begin{equation}\label{eq:scaled subcritical in plane}
\tilde u_{i}(y)=c_{n,\sigma}\lambda_{as,p_i}(K) \int_{\R^n}\frac{\hat K_i(m_i^{-(p_i-1/2\sigma)}z+y_i) \tilde u_{i}^{p_i}(z)}{|y-z|^{n-2\sigma}}\,\ud z, \quad y\in\mathbb{R}^n.
\end{equation}
Since $0<\tilde u_{i}\leq 2^{\frac{n-2\sigma}{2}}$ and $K\ge 0$, arguing exactly the same as the proof of Proposition \ref{prop:blow up a bubble} we have  $\tilde u_{i}\to u$ in
$C^{2\sigma}_{loc}(\mathbb{R}^n)$, where $u(0)=1$ and $u$ satisfies
\begin{equation}\label{eq:critical in plane1}
 u(y)=c_{n,\sigma}\lambda \int_{\R^n} \frac{K(\bar x)u(z)^{\frac{n+2\sigma}{n-2\sigma}}}{|y-z|^{n-2\sigma}}\,\ud z.
\end{equation}
Hence, $\lambda>0$, $K(\bar x)>0$, and the solutions of \eqref{eq:critical in plane1} are classified in \cite{CLO} and \cite{Li04}.

Similar to the proof of Proposition 2.1 in \cite{JLX2},
since $K\ge 0$ and $v_{i}$ are antipodally symmetric,
we have, by taking $\delta$ small,
\begin{equation}\label{norm less1}
1=\int_{\mathbb{S}^n}Kv_{i}^{p_i+1}\ge 2\int_{\mathcal{B}(\bar x,\delta)}Kv_{i}^{p_i+1}\ge 2K(\bar x)\int_{\mathbb{R}^n}u^{\frac{2n}{n-2\sigma}}+o(1),
\end{equation}
where $\mathcal{B}(\bar x,\delta)$ the geodesic ball of radius $\delta$ centered at $\bar x$ on $\Sn$. 
By the sharp Sobolev inequality \eqref{pe1} and the estimate \eqref{eq:critical in plane1}, we have
\begin{equation*}
\begin{split}
P_{\sigma}(1)\omega_n^{\frac{2\sigma}{n}}\leq \frac{\int_{\mathbb{R}^n}u(-\Delta)^{\sigma}u}{\left(\int_{\mathbb{R}^n}
u^{\frac{2n}{n-2\sigma}}\right)^{\frac{n-2\sigma}{n}}}&=\lambda K(\bar x)\left(\int_{\mathbb{R}^n}u^{\frac{2n}{n-2\sigma}}\right)^{\frac{2\sigma}{n}}\leq \lambda_{as}(K)2^{-\frac{2\sigma}{n}}(\max_{\Sn} K)^{\frac{n-2\sigma}{n}},
\end{split}
\end{equation*}
which contradicts with \eqref{eq:less1}.
\end{proof}

Next we shall prove Proposition \ref{less by test function1} using some test functions, which are inspired by \cite{Hebey, Ro}.

\begin{proof}[Proof of Proposition \ref{less by test function1}]
Let $\xi_1$ be a maximum point of $K$ at which $K$ has flatness order greater than $n-2\sigma$. Suppose $\xi_2$ is the antipodal point of $\xi_1$.
For $\beta>1$ and $i=1,2$ we define
\be\label{eq:bubble}
v_{i,\beta}(x)=\left(\frac{\sqrt{\beta^2-1}}{\beta-\cos r_i}\right)^{\frac{n-2\sigma}{2}},
\ee
where $r_i=d(x,\xi_i)$ is the geodesic distance between $x$ and $\xi_i$ on the sphere. It is clear that
\[
P_{\sigma}(v_{i,\beta})=P_{\sigma}(1)v_{i,\beta}^{\frac{n+2\sigma}{n-2\sigma}}\quad\mbox{and}\quad\int_{\mathbb{S}^n}v_{i,\beta}^{\frac{2n}{n-2\beta}}=\omega_n.
\]
Let
\[
v_{\beta}=v_{1,\beta}+v_{2,\beta},
\]
which is antipodally symmetric. Then as in the proof of Proposition 2.2 in \cite{JLX2}, we have
\[
\frac{\int_{\mathbb{S}^n}v_{\beta}P_{\sigma}(v_{\beta})}{\left(\int_{\mathbb{S}^n}Kv_{\beta}^{\frac{2n}{n-2\sigma}}\right)^{\frac{n-2\sigma}{n}}}\leq
\frac{P_{\sigma}(1)\omega_n^{\frac{2\sigma}{n}}2^{\frac{2\sigma}{n}}}{K(\xi_1)^{\frac{n-2\sigma}{n}}}
\left(1-\frac{A}{\omega_n}(\beta-1)^{\frac{n-2\sigma}{2}}+o\big((\beta-1)^{\frac{n-2\sigma}{2}}\big)\right),
\]
for $\beta$ close to $1$, \[
A=2^{-\frac{n-2\sigma}{2}}\omega_{n-1}\int_0^{+\infty}\frac{2^nr^{n-1}}{(1+r^2)^{\frac{n+2\sigma}{2}}}dr>0.
\]
This implies that \eqref{less than1} holds.
\end{proof}

Theorem \ref{K-M-E-S} can be extended to positive functions $K$ which are invariant under some isometry group acting without fixed points (see Hebey \cite{Hebey} and Robert \cite{Ro}). Denote $Isom(\mathbb{S}^n)$ as the isometry group of the standard sphere $(\mathbb{S}^n, g_{\Sn})$.
Let $G$ be a subgroup of $Isom(\mathbb{S}^n)$. We say that $G$ acts without fixed points if for
each $x\in\mathbb{S}^n$, the orbit $O_{G}(x):=\{g(x)| g\in G\}$ has at least two elements.
We denote $|O_{G}(x)|$ be the number of elements in $O_{G}(x)$. A function $K$ is called $G$-invariant
if $K\circ g\equiv K$ for all $g\in G$.

\begin{thm}\label{invariance nirenberg}
Let $n\ge 2$, $0<\sigma<n/2$,  $G$ be a finite subgroup of $Isom(\mathbb{S}^n)$ acting without fixed points.
Let $K\in C^0(\Sn)$ be a positive and G-invariant function.
If there exists $\xi_0\in\mathbb{S}^n$ such that  $K$ has flatness order greater than $n-2\sigma$ at $\xi_0$, and for any $x\in\mathbb{S}^n$
\begin{equation}\label{eq:condition 1}
\frac{K(\xi_0)}{|O_{G}(\xi_0)|^{\frac{2\sigma}{n-2\sigma}}}\geq \frac{K(x)}{|O_{G}(x)|^{\frac{2\sigma}{n-2\sigma}}},
\end{equation}
then \eqref{main equ} possesses a positive and G-invariant $C^{2\sigma^*}(\mathbb{S}^n)$ solution.
\end{thm}
The above theorem was proved by Hebey \cite{Hebey} for $\sigma=1$ and Robert \cite{Ro, Ro10} for $\sigma$ being other integers.

Let $H^{\sigma}_{G}$ be the set of G-invariant functions in $H^{\sigma}(\mathbb{S}^n)$. Let
\[
\lambda_{G}(K)=\inf_{v\in H^{\sigma}_{G}}\left\{\int_{\mathbb{S}^n}vP_{\sigma}(v): \int_{\mathbb{S}^n}K|v|^{\frac{2n}{n-2\sigma}}=1\right\}.
\]
Similar to Theorem \ref{K-M-E-S}, the proof of Theorem \ref{invariance nirenberg} is again divided into two steps.

\begin{prop}\label{less then exist}
Let $G$ be a finite subgroup of $Isom(\mathbb{S}^n)$. Let $K\in C^0(\Sn)$ be a positive and G-invariant function. If for all $x\in\mathbb{S}^n$,
\begin{equation}\label{eq:less}
\lambda_{G}(K)<\frac{P_{\sigma}(1)\omega_n^{\frac{2\sigma}{n}}|O_{G}(x)|^{\frac{2\sigma}{n}}}{K(x)^{\frac{n-2\sigma}{n}}},
\end{equation}
then there exists a positive G-invariant $C^{2\sigma^*}(\mathbb{S}^n)$ solution of \eqref{main equ}.
\end{prop}

\begin{prop}\label{less by test function}
Let $G$ be a finite subgroup of $Isom(\mathbb{S}^n)$ and act without fixed points.
Let $K\in C^0(\Sn)$ be a positive and G-invariant function. If $K$ has flatness order greater than $n-2\sigma$ at $\xi_1$ for some $\xi_1\in\mathbb{S}^n$, then
\begin{equation}\label{less than}
\lambda_{G}(K)<\frac{P_{\sigma}(1)\omega_n^{\frac{2\sigma}{n}}|O_{G}(\xi_1)|^{\frac{2\sigma}{n}}}{K(\xi_1)^{\frac{n-2\sigma}{n}}}.
\end{equation}
\end{prop}

Theorem \ref{invariance nirenberg} follows from Proposition \ref{less then exist} and Proposition \ref{less by test function} immediately. The proof of Proposition \ref{less then exist} uses subcritical approximations and blow up analysis, which is similar to that of Proposition \ref{less then exist1}. Proposition \ref{less by test function} can be verified by the following G-invariant test function
\[
v_{\beta}=\sum\limits_{i=1}^m v_{i,\beta},
\]
where $m=|O_G(\xi_1)|$, $O_{G}(\xi_1)=\{\xi_1,\dots,\xi_m\}$, $\xi_i=g_i(\xi_1)$ for some $g_i\in G$, $g_1=Id$, $v_{j,\beta}:=v_{1,\beta}\circ g_i^{-1}$ and $v_{1,\beta}$ is as in \eqref{eq:bubble}. We omit the detailed proofs of Propositions \ref{less then exist} and \ref{less by test function}, and leave them to the readers.

\subsection{Proof of the existence part of Theorem \ref{main thm A}}

Note that the interval $(n-2\sigma,n)$ is getting broader as $\sigma$ increases. We split the problem into three cases.

\medskip

\noindent \emph{Case 1.} Let $\sigma_0\in (0,1]$ and $K$ satisfy the assumptions in Theorem \ref{main thm A} for $\sigma=\sigma_0$. 

\medskip

Then $K$ also satisfies the assumptions of Theorem \ref{main thm A} for every $\sigma\in [\sigma_0,1]$. It follows from the compactness part of Theorem \ref{main thm A} that for every $\sigma\in [\sigma_0,1]$ and all positive solution of \eqref{main equ}, we have
\be\label{eq:universal bound}
C^{-1} \le v\le C,
\ee
where $C\ge 1$ depends only on $n, \sigma_0, K$ but independent of $\sigma$. It was proved by Li \cite{Li95} that
\be \label{eq:degree counting}
\deg\left(v-(P_1)^{-1}Kv^{\frac{n+2}{n-2}}, C^{2,\al}(\Sn) \cap \{C^{-1} \le v\le C\},0\right)\neq 0
\ee
for any $0<\al<1$. By Theorem \ref{thm:schauder} and a homotopy argument,  we have for all $\sigma\in [\sigma_0,1]$
\[
\deg\left(v-(P_{\sigma})^{-1}Kv^{\frac{n+2\sigma}{n-2\sigma}}, C^{2,\al'}(\Sn) \cap \{C^{-1} \le v\le C\},0\right)\neq 0
\]
for some $0<\al'<1$. Therefore, \eqref{main equ} has a positive solution for $\sigma=\sigma_0$. Since $\sigma_0$ is arbitrary, \eqref{main equ} has a positive solution for all $\sigma\in (0,1)$.

\medskip

\noindent \emph{Case 2.}  Let $\sigma_0\in [1, n/2)$ and $K$ satisfy the assumptions in Theorem \ref{main thm A} for $\sigma=\sigma_0$, and with $\beta>n-2$. 

\medskip

Then $K$ also satisfies the assumptions of Theorem \ref{main thm A} for every $\sigma\in [1, \sigma_0]$. Therefore, \eqref{eq:universal bound} still holds for all positive solution of \eqref{main equ} and for every $\sigma\in [1, \sigma_0]$, where $C\ge 1$ depends only on $n, \sigma_0, K$ but independent of $\sigma$. Moreover, \eqref{eq:degree counting} still holds, and the arguments in Case 1 still apply. Therefore, we can also conclude that \eqref{main equ} has a positive solution in the case of $\sigma>1$ and $\beta>n-2$.

\medskip

\noindent \emph{Case 3.} Let $\sigma_0\in [1, n/2)$ and $K$ satisfy the assumptions in Theorem \ref{main thm A} for $\sigma=\sigma_0$, but with $\beta\in (n-2\sigma_0, n-2]$. 

\medskip

In this case, $K$ does not satisfy the assumptions of Theorem \ref{main thm A} for all $\sigma\in [1, \sigma_0]$. If we still want to use the degree counting results of Li \cite{Li95}, we have to deform $K$ properly.

Let $\eta:[0,\infty)\to [0,1]$ be a smooth cutoff function satisfying $\eta(t)=1$ for $t\in [0,1]$, $\eta\equiv 0$ for $t\ge 2$ and  $|\eta'|\le 5$.

\begin{lem} \label{lem:deform1}
 Let $f\in C^{1}(B_1)$ satisfy
\[
f(x)=f(0)+\sum_{i=1}^n a_i|x_i|^\beta+R(x),
\]
where $a_i\neq 0$, $\beta>1$, $R(x)\in C^{[\beta]-1,1}(B_1)$ satisfies
$\sum_{s=0}^{[\beta]}|\nabla^s R(x)||x|^{-\beta
+s}\to 0$ as $x\to 0$.  For $\tau \in [0,1]$ and $\va\in [0,1/4]$, let
\[
f_{\va,\tau}(x)
:=f(x)-\tau R(x) \eta(|x|/\va).
\]
Then there exists a constant $\va_0>0$ such that $0$ is the unique critical point of  $f_{\va_0,\tau}$ in $B_{2\va_0}$ for every $\tau\in [0,1]$.
\end{lem}

\begin{proof} For $x\in \overline B_\va$, we have $f_{\va,\tau}(x)=f(0)+\sum a_i|x_i|^\beta+(1-\tau)R(x)$. It follows
from the assumption of $R$ that
$0$ is the only critical point in $\overline B_\va$ provided $\va$ is sufficiently small. For $x\in B_{2\va}\setminus B_{\va}$, we have
\[
\frac{1}{|x|^{\beta-1}}|\nabla (R(x)\eta(|x|/\va))|\le \frac{1}{|x|^{\beta-1}} |\nabla R(x)|+\frac{5}{|x|^{\beta-1}\va }|R(x)|\to 0\quad\mbox{as }\va \to 0.
\]
On the other hand, $
\frac{1}{|x|^{\beta-1}} \beta \sum_{i} |a_i| |x_i|^{\beta-1}  >c>0.
$
This completes the proof.
\end{proof}

Consequently, by Lemma \ref{lem:deform1}, for those $K$ satisfying the assumptions in Theorem \ref{main thm A}, one can make a homotopy and deform $K$ continuously to a function $K_1$ of the form \eqref{eq:K1} near each of its critical points, and along the homotopy it does not introduce new critical points and the hypotheses in Theorem \ref{main thm A} are still satisfied.

\begin{cor}\label{cor:deform1}  Let $K$ satisfy the assumptions in Theorem \ref{main thm A} for $\sigma=\sigma_0$. Then for every (small) $d>0$ there exists a continuous  $1$-parameter family of functions  $\{K_\tau\} \subset C^1(\Sn)$, $ \tau\in [0,1]$, and $\va_0>0$ such that
\begin{itemize}
\item[i).] $K_0=K$, and  $K_\tau=K$ in  $\{|\nabla K|\ge d>0\}$ for every $ \tau\in [0,1]$;
\item[ii).] For every $ \tau\in [0,1]$, $K_\tau$ has the same critical points as $K$ does, and $K_\tau$ also satisfies the assumptions for $K$ in Theorem \ref{main thm A};
\item[iii).] For every critical point of $K_1$, there exists a geodesic normal coordinates centered at it such that
\be\label{eq:K1}
K_1(y)=K_1(0)+\sum_{i=1}^n a_i|y_i|^{\beta} \quad \mbox{in }B_{\va_0},
\ee
where $\beta$ and $a_i$ are the same as those of $K$;

\item[iv)]  $K_\tau$ has uniform lower bound and $C^1$ bound in $\tau$, and $\nabla K_\tau$ has uniform modulus of continuity in $\tau$. 
\end{itemize}
\end{cor}
Recall that in \emph{Case 3}, we have $\sigma_0\in [1, n/2)$ and $\beta\in (n-2\sigma_0,n-2]$. Now we can do another homotopy of the form, roughly speaking, 
\[
K_1^{(t)}(x)=K_1(0)+\sum_{i=1}^n a_j|x_j|^{(n-1-\beta)t+\beta},\quad t\in [0,1],
\]
near each critical point of $K$, where $K_1^{(0)}=K_1$. Once again we can make such a homotopy without introducing new critical points of $K$ and keeping the hypothesis of Theorem \ref{thm:maximum bound} satisfied along the homotopy. This homotopy can be made as follows.

We let $b_1, b_2 \in C^{\infty}(\R^+)$ satisfy $b_1'\ge 0$, $b_2'\ge 0$ and
\[
b_1(t)=\begin{cases}
t,\quad & 0<t\le \va_0/8,\\
1, \quad & t \ge \va_0/4.
\end{cases}, \quad
b_2(t)=\begin{cases}
0,\quad & 0<t\le \va_0/2,\\
1, \quad & t \ge 7\va_0/8.
\end{cases}
\]
%Let $b_2(t)\in C^{\infty}(\R^+)$ satisfy $b_2'\ge 0$ and
%\[
%b_2(t)=\begin{cases}
%0,\quad & 0<t\le \va_0/2,\\
%1, \quad & t \ge 7\va_0/8.
%\end{cases}
%\]
For $z\in\R^{n-1}$, we define
\[
F(z_1,\dots, z_{n-1})=\prod_{j=1}^{n-1}(1-b_2(|z_j|)).
\]
%We also denote $(x_1,\cdots, \hat x_i, \cdots, x_n)=(x_1,\cdots, x_{i-1}, x_{i+1}, \cdots, x_n)$, i.e., the $i$-th coordinate is removed from $(x_1,\cdots, x_n)$, and let
and
\[
F_i(x)=F(x_1,\cdots, x_{i-1}, x_{i+1}, \cdots, x_n).
\]
Let $t\in [0, 1]$. By Corollary \ref{cor:deform1}, near every critical point of $K$, we deform $K_1$ (recall that $\beta\le n-2$) as
\[
K_1^{(t)}(y)= K_1(0)+\sum_{i}^n \left(a_i b_1(|y_i|)^{(n-1-\beta)t}|y_i|^{\beta} F_i(y)+a_i |y_i|^\beta (1-F_i(y))\right).
\]
It is clear that $K_1^{(0)}\equiv K_1$. Moreover, it follows that 
\begin{itemize}
\item if $|y_i|\le \va_0/8$ for every $i=1,\cdots, n$, then we have $F_i(y)=1$ and $b_1(|y_i|)=1$. Therefore,
\[
K_1^{(t)}(y)= K_1(0)+\sum_{i}^n a_i |y_i|^{(n-1-\beta)t+\beta},
\]
and for every $t$, $K_1^{(t)}(y)$ does not have any critical point other than the origin.
\item if there exists $k$ such that $|y_k|\ge 7\va_0/8$, then we have $F_j(y)=0$ when $j\neq k$, and $b_1(|y_k|)=1$. Therefore, we have
\[
K_1^{(t)}(y)= K_1(0)+\sum_{i}^n a_i |y_i|^{\beta},
\]
and for every $t$, $K_1^{(t)}(y)$ does not have an critical point.

\item if $3\va_0/8< |y_i|<7\va_0/8$ for all $i=1,\cdots, n$, we have $b_1(|y_i|)=1$ for all $i$. Therefore
\[
K_1^{(t)}(y)= K_1(0)+\sum_{i}^n a_i |y_i|^{\beta},
\]
and for every $t$, $K_1^{(t)}(y)$ does not have any critical point.

\item if there exists $k$ such that $|y_k|\le 3\va_0/8$ (which is less than $\va_0/2$), then $F_j(y)$ is independent of $y_k$ for all $j=1,\cdots, n$. Therefore, 
\[
\begin{split}
\partial_{y_k} K_1^{(t)}(y)&=\partial_{y_k} \left(a_k b_1(|y_k|)^{(n-1-\beta)t}|y_k|^{\beta} \right) F_k(y)+\partial_{y_k} \left(a_k |y_k|^\beta\right) (1-F_k(y)).
\end{split}
\]
Therefore, since $a_k\neq 0$ and $b_1'\ge 0$,  the only possible points for $\partial_{y_k} K_1^{(t)}(y)=0$ are for those $y_k=0$. Now we can consider the function $K_1^{(t)}(y)$ for $y_k=0$, where the structure of $F_i$ will stay unchanged. If we run through the above four cases $n-1$ more times, we can conclude that for all $t\in [0,1]$, the only critical point for $K_1^{(t)}(y)$ is $y=0$.

\end{itemize}
%Gluing them together, we have a homotopy $K_1^t: [0, \beta_1] \to C^1(\Sn)$. 
%It is direct to verify that $K_1^{(t)}$ satisfies all assumptions in Theorem \ref{main thm A} for $\sigma=\sigma_t:=(1-t)\sigma_0+t$ and for every $t\in [0,1]$.

On one hand, it follows from the compactness result Theorem \ref{thm:maximum bound} that for all $\sigma\in [1,\sigma_0]$ and all solutions of 
\[
P_{\sigma}(v)=c(n,\sigma)K_1^{(1)} v^{\frac{n+2\sigma}{n-2\sigma}},\quad v>0 \quad \mbox{on }\Sn,
\]
there holds
\[
C^{-1} \le v\le C,
\]
where $C\ge 1$ depends only on $n, \sigma_0, K$ but independent of $\sigma$. Therefore, it follows from \eqref{eq:degree counting} that
\be\label{eq:degree2}
\deg\left(v-c(n,\sigma_0)(P_{\sigma_0})^{-1}K_1^{(1)}v^{\frac{n+2\sigma_0}{n-2\sigma_0}}, C^{2,\al'}(\Sn) \cap \{C^{-1} \le v\le C\},0\right)\neq 0
\ee
for some $0<\al'<1$. 

On the other hand, for all $t\in[0,1]$, and all the solutions of 
\[
P_{\sigma_0}(v)=c(n,\sigma_0)K_1^{(t)} v^{\frac{n+2\sigma_0}{n-2\sigma_0}},\quad v>0 \quad \mbox{on }\Sn,
\]
there holds
\[
C^{-1} \le v\le C,
\]
where $C\ge 1$ depends only on $n, \sigma_0, K$ but independent of $t$. Therefore, it follows from \eqref{eq:degree2} that
\[
\deg\left(v-c(n,\sigma_0)(P_{\sigma_0})^{-1}K_1^{(0)}v^{\frac{n+2\sigma_0}{n-2\sigma_0}}, C^{2,\al'}(\Sn) \cap \{C^{-1} \le v\le C\},0\right)\neq 0
\]
for some $0<\al'<1$. Consequently, making use of the fact that $K_1^{(0)}\equiv K_1$, Corollary \ref{cor:deform1} and Theorem \ref{thm:maximum bound}, we have
\[
\deg\left(v-c(n,\sigma_0)(P_{\sigma_0})^{-1}Kv^{\frac{n+2\sigma_0}{n-2\sigma_0}}, C^{2,\al'}(\Sn) \cap \{C^{-1} \le v\le C\},0\right)\neq 0
\]
Therefore, we can also conclude that \eqref{main equ} has a positive solution in this last case. This finishes the proof of the existence part of Theorem \ref{main thm A}\qed

%Along the process
%\[
%K \Rightarrow K_1 \Rightarrow K_1^{\beta_1}, \quad  \sigma \Rightarrow 1,
%\]
%we have the \emph{a priori} estimate \eqref{eq:universal bound}. Consequently, we are also able to derive the existence from the degree counting.

\bigskip

\noindent T. Jin

\noindent Department of Mathematics, The University of Chicago\\
5734 S. University Avenue, Chicago, IL 60637, USA\\[1mm]
Email: \textsf{tj@math.uchicago.edu}

\medskip

\noindent Y.Y. Li

\noindent Department of Mathematics, Rutgers University\\
110 Frelinghuysen Road, Piscataway, NJ 08854, USA\\
Email: \textsf{yyli@math.rutgers.edu}

\medskip

\noindent J. Xiong

\noindent Beijing International Center for Mathematical Research, Peking University\\
Beijing 100871, China\\[1mm]
Email: \textsf{jxiong@math.pku.edu.cn}

\end{document}